\documentclass[10pt,reqno]{amsart} 
\usepackage{graphicx,subfig}
\usepackage{enumerate} 
\usepackage[colorlinks=true, pdfstartview=FitV,
linkcolor=cyan, citecolor=red, urlcolor=blue]{hyperref}
\usepackage{amsmath,amsfonts,latexsym,amssymb}
\usepackage[utf8]{inputenc} 
\usepackage[T1]{fontenc}
\usepackage{ae,aecompl}
                                                                                                                                                                                          
\usepackage{braket}
\usepackage{accents}
\newcommand*{\dt}[1]{%
  \overset{\bullet}{#1}}
  \newcommand{\dep}{\Delta^+}
\newtheorem{theorem}{Theorem}[subsection]
\newtheorem{lemma}[theorem]{Lemma}
\newtheorem{proposition}[theorem]{Proposition}
\newtheorem{corollary}[theorem]{Corollary}

\newtheorem{definition}[theorem]{Definition}

\renewcommand{\leq}{\leqslant}
\renewcommand{\geq}{\geqslant}

\renewcommand{\epsilon}{\varepsilon}

 \newcommand{\ms }{\mathsf}
\newcommand{\mk}{\mathfrak} \newcommand{\mc }{\mathcal}

\newcommand{\End}{\operatorname{End}}

\newcommand{\Kill}[2]{\langle{{#1}|{#2}}\rangle}
\renewcommand{\d}{{\rm{d}}} 
\newcommand{\ww}{\, {\pmb\wedge}\, }
\newcommand{\VS}{\vert_\Sigma } 
\newcommand{\CK}[1]{\rho\left({#1}\right)}
\newcommand{\CG}[1]{\overline{#1}}
\newcommand{\T}{\ms T} 
 
\newcommand{\X}{\ms X}

\newcommand{\ad}{\operatorname{ad}}

\newcommand\defeq
{{}\mathrel{\mathop:}={}} 
\newcommand{\eqdef}{{}=\mathrel{\mathop:}{}}
\newcommand{\rk}{\operatorname{rank}} 
\newcommand{\D}{\ms D}
\newcommand{\cwdot}{\,}
\newcommand{\MinArea}{\operatorname{MinArea}}

\newcommand{\tfe}{the following equalities in $\Omega^*(\Sigma,\mc G)$}

\title{Cyclic surfaces and Hitchin components in rank 2}
\author{Fran\c cois LABOURIE}\address{Univ. Nice Sophia Antipolis, Laboratoire J.-A. Dieudonné, Nice}
  \thanks{The research leading to these results has received funding from the European Research Council under the {\em European Community}'s seventh Framework Programme (FP7/2007-2013)/ERC {\em grant agreement} ${\rm n^o}$ FP7-246918}

\begin{document}
\begin{abstract}
  We prove that given a Hitchin representation in a split real rank 2
  group $\ms G_0$, there exists a unique equivariant minimal surface in the
  corresponding symmetric space. As a corollary, we obtain a
  parametrisation of the Hitchin component by a Hermitian bundle over
  Teichmüller space. The proof goes through introducing holomorphic
  curves in a suitable bundle over the symmetric space of $\ms G_0$. Some partial
  extensions of the construction hold for cyclic bundles in higher
  rank.
\end{abstract}

\maketitle

\tableofcontents
\section{Introduction}

We will study in this article minimal surfaces in rank 2 symmetric
spaces. More precisely, let $\ms G_0$ be a split real simple Lie group of
rank 2 and $\ms S(\ms G_0)$ be the associated symmetric space. We will
take $\ms G_0$ to be the connected component of the isometry group of
$\ms S(\ms G_0)$. The group $\ms G_0$ is in particular
isomorphic to 
$\ms{SL}(3,\mathbb R)$, $\ms{PSp}(4,\mathbb R)$ or $\ms G_{2,0}$.

Let $\Sigma$ be a connected oriented closed surface of genus greater
than 1.  We consider {\em Hitchin representations} from $\pi_1(\Sigma)$
with values in $\ms G_0$. Recall that those are deformations of {\em Fuchsian representations}; that is, discrete faithful representations in the principal $\ms{SL}_2$ in $\ms
G_0$ (See Paragraph \ref{sec:HitchinRep} for details). The {\em
  Hitchin component} $\mc H(\Sigma,\ms G_0)$ is then the space of
Hitchin representations up to conjugation by the automorphism group of
$\ms G_0$. By Hitchin \cite{Hitchin:1992es}, the Hitchin component is a smooth manifold consisting of irreducible representations. From \cite{Labourie:2006} for $\ms{PSL}(n,\mathbb R)$ (and the split groups contained in such sharing the same principal $\ms{SL}(2,\mathbb R)$) completed by  Fock and Goncharov \cite{Fock:2006a} for the remaining cases, a
Hitchin representation is discrete faithful.

Hitchin representations have a geometric interpretation. For $\ms{PSL}(2,\mathbb R)$, Hitchin representations are monodromies of hyperbolic structures, for $\ms{PSL}(3,\mathbb R)$ they are monodromies of convex real projective structures by Choi and Goldman \cite{Choi:1993vr}, in general Guichard and Wienhard have shown they  are monodromies of geometric structures  on higher dimensional compact manifolds \cite{Guichard:2012eg} The special case of $\ms{PSL}(4,\mathbb R)$ has been described by these latter authors as convex foliated projective structures in \cite{Guichard:2008cv}.

\subsection{Minimal surfaces}
One of our two main results is the following

\begin{theorem}\label{MainA}
  Given a Hitchin representation $\delta$, there exists a unique
  $\delta$-equivariant minimal mapping from $\Sigma$ to $\ms S(\ms
  G_0)$.\end{theorem}

The existence was proved by the author in \cite{Labourie:2005a}
without any assumption on the rank. The existence only relies on the fact that Hitchin representations are quasi isometric embeddings. Notice that this is not enough to guarantee uniqueness: for some quasifuchsian representations, two minimal surfaces have been constructed by Michael Anderson \cite{Anderson:1983cm} and arbitrary many by Zheng Huang and Biao Wang \cite{Huang:2012tj}.

The case of $\ms{SL}(3,\mathbb R)$ was obtained by the author in
\cite{Labourie:2006b}. The new cases are thus $\ms{Sp}(4,\mathbb R)$
and $\ms G_{2,0}$, however the proof is general. Interestingly enough,  the theorem is also valid when $\ms G_0$ is semisimple, that is $\ms G_0=\ms{SL}(2,\mathbb R)\times\ms{SL}(2,\mathbb R)$. This was done  by R. Schoen in \cite{Schoen:1993td} and see also
\cite{Bonsante:2010gk} for generalisations. 

\subsection{Parametrisation of Hitchin components}
Using the {\em Hitchin pa\-ra\-metri\-sation} of the space of minimal surfaces
\cite{Hitchin:1992es,Labourie:2005a} one obtains equivalently the following 
Theorem

\begin{theorem}\label{MainB}
  There exists an analytic diffeomorphism, equivariant under the
  mapping class group action, from the Hitchin component $\mc
  H(\Sigma,\ms G_0)$ for $\ms G_0$, when $\ms G_0$ is of real rank 2,
  to the space of pairs $(J,Q)$ where $J$ is a complex structure on
  $\Sigma$ and $Q$ a holomorphic differential with respect to $J$ of
  degree $\frac{\dim(\ms G_0)-2}{2}$.
\end{theorem}

For $\ms{SL}(3,\mathbb R)$, this corollary was obtained by Loftin in \cite{Loftin:2001} and the author
in \cite{Labourie:2006b} (announced in \cite{Labourie:BzdlFk3Q}). Theorem \ref{MainB} holds for compact surfaces; in the non compact case, the natural question is to extend the remarkable results that have
been obtained in the case of $\ms{SL}(3,\mathbb R)$, on one hand for
polynomial cubic differentials (Dumas and Wolf \cite{Dumas:2014wz}) and on
the other hand for the unit disk (Benoist and Hulin
\cite{Benoist:2013vx}. Finally, Theorem \ref{MainB} is conjectured to also hold for $\ms{SL}(2,\mathbb R)\times\ms{SL}(2,\mathbb R)$  by slightly different techniques \cite{C-L}.
\subsection{Kähler structures}
Using the theory of positive bundles and the work of Bo Berndtsson
\cite{Berndtsson:2009hr}, we extend a result obtained for cubic holomorphic differentials by Inkang Kim and
Genkai Zhang \cite{Kim:2013wc}  to
get
\begin{proposition}
  Let ${\rm m}=(m_1,\ldots,m_p)$ be a $p$-tuple of integers greater than
  1. Let $\mc E({\rm m})$ be the holomorphic vector bundle over
  Teichmüller space whose fibre at a Riemann surface $\Sigma$ is
\begin{align}
\mc E({\rm m})_\Sigma
&\defeq\bigoplus_{i=1,\ldots,p}H^0\left(\Sigma,{\mc K}^{m_i}\right),
\end{align}
where $\mc K$ is the canonical bundle of $\Sigma$.

Then the dual holomorphic bundle  ${\mc E({\rm m})^*_\Sigma}$ carries a $(p-1)$-dimensional family of mapping class group invariant Kähler
metrics, linear along the fibres and whose restriction to the zero
section is the Weil--Petersson metric.
\end{proposition}

The metric and its properties are given explicitly in Section
\ref{sec:K}. Using the (real) isomorphism between the vector bundles $\mc E(m)_\Sigma$ and $\mc E(m)^*_\Sigma$ given by the Petersson metric on the fibres, one gets:

\begin{corollary}
  The Hitchin component $\mc H(\Sigma,\ms G_0)$, when $\ms G_0$ is of
  real rank 2, carries a complex structure and a $1$-dimensional family of compatible  mapping class group invariant Kähler
  metrics for which the Fuchsian locus is totally geodesic and whose
  restriction to the Fuchsian locus is the Weil--Petersson metric.
\end{corollary}
Observe that, in this corollary, the complex structure of $\mc H(\Sigma,\ms G_0)$ is not the one inherited from the isomorphism given by Theorem \ref{MainB}. The relation of these metrics and complex structures with other
objects such as the Atiyah--Bott--Goldman symplectic form
\cite{Atiyah:1983,Goldman:1984}, the pressure metric of
\cite{Bridgeman:2013tn} or the metric exhibited by Qiongling Li
\cite{Li:2013to} for convex projective structures is rather
mysterious. However the relation with the pressure metric along the Fuchsian locus has been made explicit by Wentworth and the author in \cite{Labourie:2015wz}.
\subsection{Area rigidity} Let $\mathcal T(\Sigma)$ be the Teichmüller space of $\Sigma$. 
For a Hitchin representation $\delta$, let us define as in \cite{Labourie:2005a}
$$
{\MinArea }(\delta)\defeq\inf \{e_\delta(J)\mid J\in\mathcal T (\Sigma)\},
$$
where $e_\delta(J)$ is the energy of the unique $\delta$-equivariant harmonic map from the universal cover of $\Sigma$ equipped with $J$ to $\mathsf S(\ms G_0)$, equipped with the symmetric metric normalised so that  the principal hyperbolic plane has curvature $-1$. Motivated by a question of Anna Wienhard, we obtain
\begin{theorem}{\sc [Area Rigidity]}
The following inequality holds
\begin{align}
{\MinArea }(\delta)\geq -2\pi\chi(\Sigma).
\end{align}
Moreover, the equality holds if and only if $\delta$ is a fuchsian representation.
\end{theorem}
By Katok's Theorem \cite{Katok:1982jz}, since the intrinsic metric of any minimal surface is non positively curved, we have the inequality
\begin{align}
{\MinArea }(\delta)\cdot h(\Sigma)\geq -2\pi\chi(\Sigma).
\end{align}
Where $h(\Sigma)$ is the entropy of the induced metric on $\Sigma$ from the equivariant embedding in $\mathsf{S(G_0)}$. On the other hand $h(\Sigma)\geq h(\delta)$, where  $h(\delta)$ is the {\em entropy} of $\ms S(\ms G_0)/\delta(\pi_1(\Sigma))$ seen as the asymptotic growth of the length of closed geodesics. Thus one immediately gets
\begin{align}
{\MinArea }(\delta)\geq -\frac{2\pi\chi(\Sigma)}{h(\delta)}.
\end{align}
Thus the previous is also a consequence of the {\em entropy rigidity conjecture} for Hitchin representations, recently proved by Potrie and Sambarino \cite{Potrie:2014ut}: $h(\delta)\leq 1$ with equality if and only if $\delta$ is Fuchsian.

\subsection{Cyclic Higgs bundles}
For higher rank, we only have a very partial result. Let $m_\ell$ be the
highest exponent of $\ms G_0$. Following Baraglia \cite{Baraglia:2010vi},  let us call $\mc E_{m_\ell+1}$ the
space of {\em cyclic bundles}.  The Hitchin section gives an analytic
map $\Psi$ from the space of cyclic bundles to the Hitchin component
$\mc H(\Sigma,\ms G_0)$. We then have,

\begin{theorem} \label{Cyclicinjects} 
  The map $\Psi:\mc E_{m_\ell+1}\to \mc H(\Sigma,\ms G_0)$ is an immersion.
\end{theorem}

It would be nice to understand in a geometric way the image of $\Psi$.
\subsection{Cyclic surfaces and the idea of the proof}
The main idea of the proof is to work with {\em cyclic surfaces} which
are holomorphic curves (in a certain sense) in a bundle $\ms X$ over
the symmetric space. First we show in Section \ref{sec:CS} that the
minimal surfaces associated to Hitchin representations actually lift to $\ms X$ as cyclic surfaces, and
conversely every projection of a cyclic surface is minimal. This is
strongly related to a work of Baraglia \cite{Baraglia:2010vi} and
cousin to a construction by Bolton, Pedit and Woodward in
\cite{Bolton:1995vx}. Then, complexifying the situation and treating
cyclic surfaces as solutions to a Pfaffian system, the core of the
proof is to prove an infinitesimal rigidity result for cyclic surfaces
in Section \ref{sec:IR}. To conclude, we use in Section
\ref{sec:Proper} the main result of \cite{Labourie:2005a}.

\vskip 1 truecm This paper would not have existed without numerous and
illuminating discussions with Mike Wolf. I thank him very deeply here. I also wish to thank Jean-Benoît Bost, Christophe Mourougane and
Nessim Sibony for references.  I also want to thank Anna Wienhard for fruitful questions, Brian Collier,  Steve Kerckhoff, Qiongling Li, Marco Spinaci and Richard Wentworth for their interest, comments and  important corrections on an earlier version. I also warmly thanks the anonymous referees for the numerous improvements that they suggested and the genuine mistakes they made me correct.
\section{Lie theory preliminaries}

In this section, we review for the convenience of the reader, without proof, roots system theory with
applications to the construction of the maximal compact subgroup, and
that of the split form of a complex simple Lie group.

We explain that the choice of a Cartan subalgebra, a system of positive roots
 and a Chevalley system define naturally two commuting anti
linear automorphisms of the Lie algebra, whose fixed points are
respectively the maximal compact subgroup, and the maximal split real form. We study the basic properties of these objects.

We also introduce the cyclic roots set and prove Proposition
\ref{crucial} that will play a central rôle in the proof.

The material comes form Baraglia \cite{Baraglia:2010vi}, Kostant
\cite{Kostant:1959wi}, Hitchin \cite{Hitchin:1992es}, Bourbaki
\cite{Bourbaki:owAvyv1m}. The only non standard material (which is
probably common lore) is in Section \ref{sec:braroot}.

We will use the following typographic convention: for any Lie group
$\ms H$, we shall denote by $\mk h$ its Lie algebra.

\subsection{Roots}\label{sec:PosRoot}

We recall the notations and basic facts of the root system theory
\begin{itemize}
	\item Let $\ms G$ be a complex simple Lie group, with Lie algebra $\mk g$.
	\item Let $\ms H$ be a maximal abelian semisimple
subgroup and $\mk h$ its Lie algebra called the {\em Cartan subalgebra}.
   \item We denote by $\Kill{.}{.}$ the 
  Killing form of $\mk g$, and  by the same symbol  the 
restriction of the 
  Killing form to $\mk h$ and its dual extension to $\mk h^*\defeq \hom(\mk
h,\mathbb C)$.
\item Let $\Delta\subset\mk h^*$ be the  set of roots of $\mk h$, let $\Delta^+$ be a choice of positive roots and $\Pi$ the associated simple roots.
\item We denote by $\mk g_\alpha$ the root space of the root $\alpha$ and recall that $\dim(\mk g_\alpha)=1$. The coroot of $\alpha$ is denoted by $h_\alpha$ and satisfies 
$$
\Kill{h_\alpha}{u}=\frac{2}{\Kill{\alpha}{\alpha}}\alpha(u).
$$ The {\em root space decomposition} is
  \begin{align}
    \mk g&=\mk h \oplus\bigoplus_{\alpha\in\Delta}\mk
    g_\alpha.\label{fonddec}
  \end{align}
\item Given a root $\displaystyle{
  \beta=\sum_{\alpha\in\Pi}n_\alpha\cwdot\alpha}$,
the {\em degree} of $\beta$ is $\deg(\beta)\defeq\sum_{\alpha\in\Pi}n_\alpha$.
\item 
The {\em longest root} is the uniquely defined positive root
$\eta$ so that for any positive root $\beta$, $\beta+\eta$ is not a
root.
\item A {\em Chevalley base} (\cite{Bourbaki:owAvyv1m} Ch. VIII, §2, n.4, Def. 3),  is a collection of non zero vectors $\{{\rm
  x}_\alpha\}_{\alpha\in\Delta}$ in $\mk g$ so that
  \begin{enumerate}
  \item there exist integers $N_{\alpha,\beta}$ such that for all roots $\alpha$ and $\beta$
  \begin{align*}
  {\rm x}_\alpha &\in \mk g_\alpha\cr [{\rm x}_\alpha,{\rm
    x}_{-\alpha}]&=-{\rm h}_\alpha, \cr [{\rm x}_\alpha,{\rm
    x}_\beta]&=N_{\alpha,\beta}{\rm x}_{\alpha+\beta},
\end{align*}
where by convention $\rm x_{\alpha+\beta}=0$ if $\alpha+\beta$ is not a root.
\item the
antilinear endomorphism which sends ${\rm
  x}_\alpha$ to ${\rm x}_{-\alpha}$ and is $-{\rm id}$ on 
 $\mk h$  is an antilinear automorphism of
$\mk g$. \label{it:2cb}
  \end{enumerate}

\end{itemize}

\subsection{Cyclic roots and projections}
Let $\ms G$, $\Delta$, $\Delta^+$ and $\Pi$ be as above and $\eta$ the
longest positive root.  Recall that we have
$$
\forall \alpha\in\Delta^+: \ \ \deg(\alpha)\leq\deg(\eta),
$$
with  equality only if $\eta=\alpha$.

\begin{definition}{\sc[Cyclic root sets]}\label{sec:cycroot}

  \begin{enumerate}
 
  \item The {\em conjugate cyclic root set} is
    \begin{equation}
      Z^\dagger\defeq \Pi\cup\{-\eta\}.
    \end{equation}
  \item The {\em cyclic root set} is
    \begin{equation}
      Z\defeq \{\alpha\in \Delta\mid -\alpha\in Z^\dagger\}.
    \end{equation}
  \end{enumerate}
\end{definition}
Observe that $ \Delta\setminus\left(Z\sqcup Z^\dagger\right)=\{\alpha,
\vert\deg(\alpha)\vert\not\in\{1,\deg(\eta)\}\} $.

\subsubsection{Projections}

We consider the following projections (whose pairwise disjoint product
are zero) from $\mk g $ to itself that comes from the decomposition
\eqref{fonddec}
\begin{align}
  \pi_0: \mk g  &\to \mk h,\\
  \pi:\mk g &\to \mk g_{Z}\defeq  \bigoplus_{\alpha\in Z}\mk g_\alpha,\\
  \pi^\dagger: \mk g  &\to\mk g_{Z^\dagger}\defeq  \bigoplus_{\alpha\in Z^\dagger}\mk g_\alpha,\\
  \pi_1: \mk g  &\to\mk g_1\defeq  \bigoplus_{\alpha\not\in Z\cup Z^\dagger}\mk g_\alpha.\\
\end{align}
Obviously
$$
\pi+\pi^\dagger+\pi_0+\pi_1=\rm{Id}.
$$
Moreover these projections are Killing orthogonal since the root space decomposition is Killing orthogonal.
\subsubsection{Brackets of cyclic roots}\label{sec:braroot}

The main observation about this decomposition is the following trivial
but crucial observation
\begin{proposition}{\sc[Brackets]}\label{crucial}
  We have
  \begin{align}
    [\mk h,\mk g_{Z^\dagger}]&\subset\mk g_{Z^\dagger},\label{pidec1}\\
    \,[\mk h,\mk g_{Z}]&\subset\mk g_{Z},\label{pidec2}\\
    \,[\mk h,\mk g_1]&\subset\mk g_1,\label{pidec3}\\
    \,[\mk g_Z,\mk g_{Z^\dagger}]&\subset\mk h,\label{pidec4}\\
         [\mk g_{Z^\dagger},\mk g_1]&\subset\mk g_1\oplus\mk g_{Z},\label{pidec0}\\
    \,[\mk g_{Z},\mk g_1]&\subset\mk g_1\oplus\mk g_{Z^\dagger},\label{pidec7}\\
 \,[\mk g_Z,\mk g_Z]&\subset \mk g_1\oplus\mk g_{Z^\dagger},\label{pidec8}\\
\,[\mk g_{Z^\dagger},\mk g_{Z^\dagger}]&\subset \mk g_1\oplus\mk
    g_{Z}\label{pidec9}.
  \end{align}
\end{proposition}
\begin{proof} In this proof, we write for any $w\in\mk g$
$$
w=w_0+\sum_{\alpha\in\Delta}w_\alpha,
$$
where $w_0\in\mk h$ and $w_\alpha\in\mk g_\alpha$. Observe that
Assertions \eqref{pidec1},\eqref{pidec2} and \eqref{pidec3} follows
from the fact that
$$
[\mk h,\mk g_\alpha]\subset \mk g_\alpha.
$$ 
We will now use the following two facts in the proof
\begin{enumerate}
\item If $\alpha$ and $\beta$ are distinct simple roots, then $
\alpha-\beta
$
is not a root.
\item If $\alpha$ is a positive root and $\eta$ the longest root then
  $\alpha+\eta$ is not a root.
\end{enumerate}

Combining the two, we get that if $\alpha\in Z$ and $\beta\in
Z^\dagger$ then $[\mk g_\alpha, \mk g_\beta]\not=\{0\}$  if and only if
$\alpha+\beta=0$. Thus if $v\in \mk g_Z$ and $u\in\mk g_{Z^\dagger}$
then
$$
[u,v]=[u_\eta,v_{-\eta}]+\sum_{\beta\in \Pi}[u_\beta,v_{-\beta}]\in\mk
h
$$
This proves \eqref{pidec4}.

Let $\alpha$ be a simple root and $\gamma\not\in
Z\cup Z^\dagger$ a root of degree $a$, in particular the degree of
$\alpha+\gamma$ (if it is a root) is $a+1$. Then
\begin{itemize}
\item Since $a\not=-1$, then $a+1\not=0$, thus $\pi_0\left([\mk g_\alpha,\mk
  g_\gamma]\right)=\{0\}$,
\item Since $a\not=0$, then $a+1\not=1$, moreover
  $a+1\not=-\deg(\eta)$, thus $\pi_{Z^\dagger}([\mk g_\alpha,\mk g_\gamma])=\{0\}$.
\end{itemize}
Thus $[g_\alpha,g_\gamma]\subset \mk g_1\oplus \mk g_{Z}$.
Similarly
\begin{itemize}
\item Since $a\not=\deg(\eta)$, then $a-\deg(\eta)\not=0$, thus $\pi_0([\mk
  g_{-\eta},\mk g_\gamma]) =\{0\}$,
\item Since $a\not=0$, then $a-\deg(\eta)\not=-\deg(\eta)$, moreover
  $a-\deg(\eta)\not=1$, thus $\pi_{Z^\dagger}([\mk g_{-\eta},\mk g_\gamma])
  =\{0\}$.
\end{itemize}
Thus $[g_{-\eta},g_\gamma]\subset \mk g_1\oplus \mk g_{Z}$ and combining with the previous assertion, we obtain that $[g_{Z^\dagger},g_\gamma]\subset \mk g_1\oplus \mk g_{Z}$.
This finishes proving Assertion \eqref{pidec0}. Assertion
\eqref{pidec7} follows by symmetry.

Finally, Assertion \eqref{pidec8} follows from the fact that if
$\alpha$ and $\beta$ belong to $Z$, then $\alpha+\beta\not\in Z$: if
$\alpha$ and $\beta$ are both simple, then $\alpha+\beta$ is not
simple and positive, if $\alpha$ is simple and $\beta=-\eta$, then
$\alpha+\beta$ is negative and not the longest. Assertion
\eqref{pidec9} follows by symmetry.
\end{proof}

\subsection{The principal 3-dimensional
  subalgebras}\label{sec:Principal}

We begin by recalling the existence and properties of the principal
$\mk{sl}(2,\mathbb C)$. Let $a$ in $\mk h$ and $r_\alpha$ in $\mathbb R$ be defined by
\begin{align*}
  a\defeq\frac{1}{2}\sum_{\alpha\in\Delta^+}{\rm
    h}_\alpha\eqdef\sum_{\alpha\in\Pi}r_\alpha\cwdot {\rm h}_\alpha.
\end{align*}
Let now
\begin{align*}
  X\defeq \sum_{\alpha\in\Pi}\sqrt{r_\alpha} {\rm x}_\alpha,\ Y\defeq
  \sum_{\alpha\in\Pi}\sqrt{r_\alpha} {\rm x}_{-\alpha}.
\end{align*}
Then, we have
\begin{proposition}{\sc[Kostant]}\label{prop:Kostant}
  The span of $(a,X,Y)$ is a subalgebra $\mk s$ isomorphic to
  $\mk{sl}(2,\mathbb C)$ so that
$$
[a,X]=X,\ \ [a,Y]=- Y,\ \ [X,Y]=-a.
$$
Moreover for any root $\alpha$, we have
\begin{align}
  \deg(\alpha)&=\alpha(a).
\end{align}
\end{proposition}

We remark that we follow here a mixture of  the conventions by Kostant and Bourbaki on the canonical basis for the Lie algebra of $\mk
{sl}_2$: no factor of 2 (Kostant), $[X,Y]=-a$ (Bourbaki). In our case $Y=-X^t$, although a very classical convention (actually Kostant's) is to have $Y=X^t$.

\begin{definition}{\sc[Principal subalgebras]}

  \begin{itemize}
  \item A 3-dimensional subalgebra $\mk s$ isomorphic to $\mk{sl}_2$
    is a {\em principal subalgebra} if it contains an element
    conjugate to $X$ (and then an element conjugate to $a$).
  \item A principal subalgebra is an {\em $\mk h$-principal
      subalgebra}, if it intersects non trivially the Cartan subalgebra $\mk h$ and the intersection contains the sum $a$ of the positive coroots.
  \item A {\em principal $\ms{SL}_2$} in a complex semisimple group is a
    group whose Lie algebra is a principal subalgebra.
  \item A {\em principal $\ms{SL}_2$} in a split real group is a split real group
    whose complexification is a principal $\ms{SL}_2$.
  \end{itemize}
\end{definition}

As an example, the Lie algebra $\mk s$ generated by $(a,X,Y)$ is an
$\mk h$-principal subalgebra of $\mk g$. We then have from Theorem
4.2 in \cite{Kostant:1959wi},
     
      \begin{proposition}\label{UniquePrincipal}
        Any two principal subalgebra are conjugate, and moreover any
        two $\mk h$-principal subalgebras are conjugated
        by an element of $\ms H$.
      \end{proposition}

\subsubsection{Exponents and decomposition under the principal subalgebra.}
We use the notation of the previous paragraph. Let $\ell\defeq \rk(\ms
G)$. Then have

\begin{proposition}{\sc[Kostant]}\label{pro:kline} There exists a increasing sequence
  of integers $\{m_1,\ldots, m_{\ell}\}$ called the {\em exponents} of
  $\ms G$ such that the Lie algebra $\mk g$ decomposes as the sum of
  $\ell$ irreducible representations of $\mk s$
$$
\mk g=\sum_{i=1}^\ell \mk v_i,
$$
with $\dim(\mk v_i)=2m_i+1$. Moreover $\mk v_1=\mk
s$.
Finally, if $\mk s$ is an $\mk h$-principal subalgebra, we have the {\em Kostant splitting}
$$
\mk h=\bigoplus_{i=1}^{\ell} \mk a_i,
$$
where $\mk a_i=\mk h\cap \mk v_i$ are of dimension 1 and called the {\em Kostant lines}.

\end{proposition}

Observe that $m_1=1$. For a rank 2 group, we furthermore have
$$
m_2=\frac{\dim(\ms G)-4}{2}
$$
Thus for $\ms G_{2,0}$, $m_2=5$, for $\ms{Sp}(4,\mathbb R)$, $m_2=3$, for
$\ms{SL}(3,\mathbb R)$, $m_2=2$.

Fixing now a Chevalley base and the associated generators $(a,X,Y)$ of
$\mk s$, we recall that a {\em highest weight vector} in $\mk v_i$ is
an eigenvector $e_i$ of $a$ satisfying $[e_i,X]=0$. The highest weight
vectors in $\mk v_i$ generate a line. Observe that we have
\begin{proposition}\label{etaHiggs} Let $\eta$ be the longest root.
  Then \begin{align} m_\ell&=\deg(\eta),
  \end{align}
  and ${\rm x}_\eta$ is a highest weight vector in $\mk v_\ell$.
\end{proposition}

\subsection{The maximal compact subgroup and the first
  conjugation}\label{realcond}

Let $\mk h$ be a Cartan subalgebra. An {\em $\mk h$-Cartan involution} is 
an antilinear Lie algebra automorphism $\rho$, globally preserving $\mk h$, satisfying $\rho^2=1$ and 
so that $(u,v)\mapsto-\Kill{u}{\rho(v)}$ is  positive definite. A {\em Cartan involution} is an $\mk h$-Cartan involution for some $\mk h$.

We finally define the {\em symmetric space} of $\ms G$ 
to be the space $\ms S(\ms G)$ of Cartan involutions.

From \cite{Bourbaki:owAvyv1m} and specifying Property \ref{it:2cb} of Chevalley basis, we have
\begin{proposition}\label{UniqueCompact}
  Given a Chevalley base $\{{\rm x}_\alpha\}_{\alpha\in \Delta}$,
  there exists a unique $\mk h$-Cartan involution $\rho$ so that
  \begin{align}
    \rho({\rm x}_\alpha)&={\rm x}_{-\alpha}.
  \end{align}
  Moreover this Cartan involution preserves the principal subalgebra
  associated to the Chevalley basis as in Proposition
  \ref{prop:Kostant}.
\end{proposition}

Observe that we follow the convention of Bourbaki and other authors may have different conventions due to different conventions on defining Chevalley bases.
The following proposition summarises some useful properties.

\begin{proposition}\label{UniqueCartan}\label{SymSpace}\label{geom-sym}
  Any two $\mk h$-Cartan involutions are conjugated by an element of
  $\ms  H$.  An $\mk h$-Cartan involution preserves $\mk h$, send roots to
  opposite roots, coroots $h_\alpha$ to their opposite, and $\mk g_1$ ({\it cf.} Section \ref{sec:cycroot}) to
  itself.

The group $\ms G$ acts transitively on $\ms S(\ms G)$, and the
  stabiliser at a Cartan involution $\rho$ is a maximal
  compact subgroup $\ms K$ so that $\ms S(\ms G)$ is isomorphic as a
  homogeneous $\ms G$-space with $\ms G/\ms K$.
\end{proposition}

The symmetric space is equipped with an interesting geometry. Let $\mc G$ be the trivial $\mk g$ bundle over $\ms{S(G)}$ with its trivial connection $D$. We define the {\em Maurer--Cartan} form $\omega\in\Omega^1(\ms{S(G),\mc G})$ as the identification of $\T_\rho(\ms{S(G)})$ with $\mk p=\{u\mid \rho(u)=-u\}$. Observe also that by construction we have a section $\rho$ of $\operatorname{Aut}(\mc G)$ such that $\rho(x)=x$ where in the right term $x$ is considered as an involution of $\mk g$. We thus have a Riemannian metric $g$ on $\mc G$, defined by $g(u,v)=-\Kill{u}{\rho(v)}$.

\subsection{The split real form and the second
  conjugation}\label{realsplit}

In this subsection we review the construction of the maximal split real
form following Kostant \cite{Kostant:1959wi}. We will prove

\begin{proposition}\label{sig-involution}  Let $\mk h$ be a Cartan subalgebra equipped with a system of positive roots. Let $\rho$ be an $\mk h$-Cartan involution preserving an $\mk h$-principal subalgebra $\mk s$. Then there exists a unique linear
  involution $\sigma$ of $\mk g$ with the following properties.
  \begin{enumerate}
  \item The involution $\sigma$ is an automorphism of $\mk g$ that preserves globally and non trivially 
   $\mk s$, such that $\left.\sigma\right\vert_{\mk a_i}=(-1)^{m_i+1}$.
     \item The involution $\sigma$ commutes with $\rho$.
  \item The set of fixed points of $\sigma\circ\rho$ is the Lie algebra
    of the split real form $\ms G_0$ of $\ms G$.
  \item if $\eta$ is the longest root then $\sigma({\rm h}_\eta)={\rm
      h}_\eta$.
  \item the automorphism $\sigma$  preserves globally $\mk g_Z$ and $\mk g_{Z^\dagger}$.
  \end{enumerate}
\end{proposition}
The Kostant lines $\mk a_i$ are defined in Proposition \ref{pro:kline}. 
The first condition actually implies that if $(a,X,Y)$ is the basis constructed of $\mk s$ from the Cartan involution $\rho$, then $\sigma(a,X,Y)=(a,-X,-Y)$.
\subsubsection{Existence of $\sigma$}\label{constructsig}

We sketch the construction of $\sigma$ due to Kostant and prove the
existence part of Proposition \ref{sig-involution} . Let $\mk h$ be a
Cartan subalgebra with its set of roots $\Delta$, set of positive
roots $\Delta^+$ and set of simple roots $\Pi$. We also fix a
Chevalley base $\{{\rm x}_\alpha\}_{\alpha\in\Delta}$ and denote by $\mk s$ the associated principal subalgebra.

Let $\mk z=\ker\left(\operatorname{ad}(X)\right)$, observe that $\mk z$ is the vector space spanned by the highest weight
vectors for the action of $\mk s$ for all the  $\mk v_i$. Observe that $\mk z$ and $\ad(Y)$ generates
$\mk g$ as a Lie algebra.  This data thus defines an involution $\sigma$ on $\mk g$
characterised by
\begin{align}
\left.\sigma\right\vert_{\mk z}=-1, \ \ \sigma(Y)=-Y.\label{def-sig}
\end{align}
Let now $\rho$ be the unique Cartan involution associated to $\mk h$
and the choice of the
Chevalley basis according to Proposition \ref{UniqueCompact}.  From \cite{Hitchin:1992es}, we gather that
\begin{proposition}\label{hitchin-sig}{\sc [Hitchin involution]}
  \noindent
  \begin{enumerate}
  \item The involution $\sigma$ and the antilinear involution
    $\rho$ commute.
  \item The set of fixed points $\mk g_0$ of the antilinear
    involution $\sigma\circ\rho$ is the Lie algebra of a split real
    form $\ms G_0$ of $\ms G$.
  \item The set of fixed points of $\sigma$ is the complexification of
    the Lie algebra of the maximal compact subgroup of $\ms G_0$.
  \item Finally, the elements $a_i\defeq\ad(Y)^{m_i} e_i$, where $m_i$ are
    the exponents and $e_i$ highest weight vectors of $\mk v_i$,
   is basis of $\mk h$ and $\sigma(a_i)=(-1)^{m_i+1}a_i$.
  \end{enumerate}
\end{proposition}
Observe that given $\mk s$ and $\mk h$, the vector $a_i$ generates the Kostant line $\mk a_i$. A corollary of the last statement and standard results about $\mk{sl}_2$ modules is
\begin{proposition}\label{sh}
  The union of $\mk h$ and $\mk s$ generates $\mk g$ as a Lie algebra.
\end{proposition}

We will usually write for any $u\in\mk g$,
$$
\lambda(u)\defeq \sigma\circ\rho(u).
$$

We finally observe the following fact that concludes the proof of the
existence part in Proposition \ref{sig-involution} and that we could not find in the literature.
\begin{proposition}{\sc [Involution and the longest root]}\label{pro:long-sig}
  Let $\eta$ be the longest root.  Then
  \begin{align}
    \sigma({\rm x}_\eta)=-{\rm x}_\eta, &\,\,\,\, \sigma({\rm h}_\eta)={\rm
      h}_\eta.
  \end{align}
  Moreover $\sigma$ globally preserves $\mk g_Z$ and 
  $\mk g_{Z^\dagger}$. \end{proposition}

\begin{proof}
  Since for each positive root $\alpha$, $\alpha+\eta$ is not a root
  and thus $[{\rm x}_\alpha,{\rm x}_\eta]=0$, we get that
  \begin{align}
    [{\rm x}_\eta,X]&=\sum_{\alpha\in\Pi}\sqrt{r_\alpha}[{\rm
      x}_\alpha,{\rm x}_\eta]=0,\cr [a,{\rm x}_\eta]&=\deg(\eta).{\rm
      x}_\eta.
  \end{align}
  In particular, ${\rm x}_\eta$ generates an irreducible
  representation of $\mk s$ for which ${\rm x}_\eta$ is a highest
  weight vector. In particular, $\sigma({\rm x}_\eta)=-{\rm x}_\eta$. Then since $\sigma$ commutes with $\rho$ we have
$$
\sigma({\rm x}_{-\eta})=\sigma(\rho({\rm x}_\eta))=\rho(\sigma({\rm
  x}_\eta))=-\rho({\rm x}_{\eta})=-{\rm x}_{-\eta}.$$ Moreover,
$\sigma$ being an automorphism of the Lie algebra, we have that
$$
\sigma({\rm h}_\eta)=-\sigma([{\rm x}_\eta,{\rm
  x}_{-\eta})])=-[\sigma({\rm x}_\eta),\sigma({\rm
  x}_{-\eta})]=-(-1)^2[{\rm x}_\eta,{\rm x}_{-\eta}]={\rm h}_\eta.
$$

Finally, $\sigma$ is an automorphism of $\mk g$ preserving $\mk h$, thus sending roots to roots. Since $\sigma$ preserves the longest coroot, it preserves the degree (using the Killing form) and thus sends simple roots to simple roots. \end{proof}
The next proposition follows from the remark following the definition of $\sigma$,
\begin{proposition}\label{prop:lS}{\sc [Involution and the principal subgroup]}
We have
   \begin{align}
    \lambda(Y)=-X,\,\,\,
    \lambda(X)=-Y,\,\,\,
    \lambda(a)=-a.
  \end{align} 
\end{proposition}

\subsubsection{Uniqueness of $\sigma$}
We now prove the uniqueness part of Proposition \ref{sig-involution} 
\begin{proposition}\label{UniqueSigma}
  Let $\mk h$ equipped with a system of positive roots  and $\mk s$ be an
  $\mk h$-principal subalgebra. Then there exists a unique  linear
  involution $\sigma$ such that
  \begin{enumerate}
  \item $\sigma$ is an automorphism of $\mk g$,
  \item $\sigma$ preserves globally $\mk s$ and $\mk h$, and is non trivial on $\mk s$
  \item $\sigma$ fixes pointwise $\mk s\cap \mk h$,
  \item $\sigma(a_i)=(-1)^{m_i+1}a_i$, 
  \end{enumerate}
\end{proposition}
\begin{proof} If $\sigma_1$ and $\sigma_2$ are two such
  involutions. Let $I=\sigma_1\circ\sigma_2$. Then $I$ fixes
  $\mk h$ pointwise. Since $I$ fixes $\mk s$ pointwise by the remark following Proposition \ref{sig-involution}, it follows
  that $I$ fixes the Lie algebra generated by $\mk h$ and $\mk s$, that
  is, $\mk g$ by Proposition \ref{sh}.\end{proof}

\section{Hitchin--Kostant quadruples}\label{Hitchinquadruple}

In this section, we introduce the basic algebraic concept used by Hitchin in the construction of the {\em Hitchin section} that we explain in the next section.
\subsection{Definitions}
Let $\ms G$ be a complex simple group.
\begin{definition} A {\em Hitchin--Kostant quadruple} is a quadruple $(\mk h,\Delta^+,
  \rho,\lambda)$ where \begin{itemize}
  \item $\mk h$ is a Cartan subalgebra equipped with a system of positive roots $\Delta^+$,
  \item $\rho$ is an antilinear involution globally fixing $\mk h$, and 
    whose set of fixed points is the Lie algebra of a maximal compact
    subgroup,
  \item $\lambda$ is an antilinear involution commuting with $\rho$,
    globally fixing $\mk h$, and  whose set of fixed points is the Lie
    algebra of a maximal split real form,
  \item $\rho$ and $\lambda$ both fix globally the same $\mk h$-principal subalgebra $\mk s$ and $\rho\circ\lambda$ and more over $\rho(a_i)=(-1)^{m_i+1}\lambda(a_i)$ for  $a_i$ in the Kostant line $\mk a_i$.  \end{itemize}
\end{definition}
We will also
use the notation $\sigma\defeq\lambda\circ\rho$.
By Proposition \ref{sig-involution}, Hitchin--Kostant quadruples exist.
Moreover:
\begin{proposition}\label{UniqueHitchinquadruple}
  Any two Hitchin--Kostant quadruples are conjugate.
\end{proposition}

\begin{proof} 
  Let $(\mk h_1,\Delta_1^+,\rho_1,\lambda_1)$ and $(\mk
  h_2,\Delta_2^+,\rho_2,\lambda_2)$ be two Hitchin--Kostant quadruples. Let $\mk s_i$ be the $\mk h_i$-principal
  subalgebra fixed globally by $\rho_i$ and $\lambda_i$. Let
  $\sigma_i=\lambda_i\circ\rho_i$.  By Proposition
  \ref{UniquePrincipal}, we can as well assume after conjugation that

$$
(\mk h_1,\mk s_1)=(\mk h_2,\mk s_2)\eqdef (\mk h,\mk s).
$$
Thus by Proposition \ref{UniqueSigma}, $\sigma_1=\sigma_2$. Applying
Proposition \ref{UniqueCartan} to $\mk s$, we can further use a
conjugation by an element of $\exp(\mk h\cap \mk s)$ so that the restrictions
of $\rho_1$ and $\rho_2$ coincide on $\mk s$. Thus $\rho_1\circ\rho_2$
is the identity on $\mk s$ and is also the identity on $\mk h$ by
Proposition \ref{UniqueCartan} (two $\mk h$-Cartan involutions have the same restriction on $\mk h$). Since $\mk s$ and $\mk h$ generate
$\mk g$ by Proposition \ref{sh}, it follows that $\rho_1=\rho_2$, thus
$\lambda_1=\lambda_2$ and the result follows.
\end{proof}

\subsection{The stabiliser of a Hitchin--Kostant quadruple}
Let $(\mk h,\Delta^+,
  \rho,\lambda)$ be a Hitchin--Kostant quadruple. Let $\ms K$ be 
the maximal compact subgroup in $\ms G$ stabilising  $\rho$  and
similarly $\ms G_0$ the split real form stabilising $\lambda$ in $\ms G$.
From Section 6 in Hitchin \cite{Hitchin:1992es}, we have
\begin{proposition}{\sc [Hitchin]}\label{HitchinTorus} The
  group $\ms K_0\defeq \ms K\cap \ms G_0$ is a maximal compact
  subgroup of $\ms G_0$.  The algebra $\mk t=\mk g_0\cap\mk h\cap\mk k$
  is  the Lie algebra of  maximal torus $\ms T$ of $\ms K_0$.
\end{proposition}

Observe that $\ms G$ acts by conjugation on the space $\X $ of Hitchin
quadruples.

\begin{proposition}\label{X}
  The stabiliser in $\ms G$ of the Hitchin--Kostant quadruple $(\mk h,\dep,\rho,
  \lambda)$ is $\ms T$.
\end{proposition}

\begin{proof} The normaliser of $\mk h$ equipped with a system of positive roots is $\ms H$. The normaliser of $\lambda$ is included in $\ms
  G_0$ since it normalises the set of fixed points of $\lambda$.  The
  normaliser of $\rho$ is similarly included in $\ms K$. The result
  follows.
\end{proof}

\begin{proposition}\label{torus}
We have the following
\begin{enumerate}
\item \label{Heta}
  We have ${\rm h}_\eta\in \mk t_\mathbb C$,
\item \label{sl3}
For $\ms G=\ms{SL}(3,\mathbb C)$, for every 
  simple root $\alpha$, $ {\rm h}_\alpha\not\in \mk t_\mathbb C$.
  \item \label{sig-involutionG2} If $\ms G$ is $\ms G_{2}$ or
  $\ms{Sp}(4,\mathbb C)$ Then $\left.\sigma\right\vert_\mk h=1$, or in
  other words $\mk h=\mk t_\mathbb C$.
\end{enumerate}

\end{proposition}

\begin{proof} 
We first prove Statement \eqref{Heta}. Since $\sigma({\rm h}_\eta)={\rm h}_\eta$ by Proposition
  \ref{pro:long-sig}, and $\mk t_\mathbb C$ is the set of fixed
  points of $\sigma$ in $\mk h$ statement follows.

We now prove Statement \eqref{sig-involutionG2}. In this case,
  since $\ms G$ is of rank 2, we have only two representations of $\mk
  s$ appearing in $\mk g$. Let ${\rm x}_1$ and ${\rm x}_2$ be the
  corresponding highest weight vectors. In that case the exponent
  $m_1$ and $m_2$ are both odd. Thus let $a_i=\ad(Y)^{m_i} {\rm x}_i$,
  then
$
\sigma(a_i)=a_i$.
But the $a_i$ generate $\mk h$ as a vector space by the last assertion of Proposition
\ref{hitchin-sig}.

We finally prove Statement \eqref{sl3}.
  When $\ms G=\ms{SL}(3,\mathbb C)$, the two exponents have different parity, thus  $\mk t_\mathbb C$ is of dimension 1;
  it follows from Property \eqref{Heta} that $\mk t_\mathbb C={\rm
    h}_\eta.\mathbb C$. Since for all simple root $\alpha$, ${\rm
    h}_\alpha$ is not collinear to ${\rm h}_\eta$, the result follows.
\end{proof}

\section{The space of Hitchin--Kostant quadruples}

Our goal now is to describe the geometry of the homogeneous $\ms
G$-space  $\X$ of Hitchin--Kostant quadruples. In particular we wish to describe a Lie
algebra bundle over $\X $ which come equipped with a connection, a
metric and other differential geometric devices.
\subsection{Preliminary: forms with values in a Lie algebra bundle}\label{prel-form}

We shall in the sequel study forms on a manifold $M$ with values in a
Lie algebra bundle $\mathcal G$. We store in this paragraph the formulas that we shall use later.

We denote by $\Omega^*(M,\mathcal G)$
the graded vector space of forms on $M$ with values in $\mc G$. We say
a form $\alpha$ in $\Omega^*(M,\mathcal G)$ is {\em decomposable} if
$\alpha=\widehat\alpha\otimes A$ where $A$ is a section of $\mathcal G$
and $\widehat\alpha$ a form on $M$. We recall the existence of a unique linear binary operation $\ww$
$$
\Omega^p(M,\mc G)\otimes\Omega^q(M,\mc G)\to\Omega^{p+q}(M,\mc G),
$$
so that if $\alpha=\widehat\alpha\otimes A$,
$\beta=\widehat\beta\otimes B$ are decomposable forms then
\begin{align}
  \alpha\ww\beta=\left(\widehat\alpha\wedge\widehat\beta\right)\otimes
  [A,B].
\end{align}
Similarly, if $\alpha$ and $\beta$ are $\mc G$ valued forms, $\Kill{\alpha}{\beta}$ is the form defined for decomposable forms $\alpha=\widehat\alpha\otimes A$ and 
$\beta=\widehat\beta\otimes B$ by
\begin{align*}
\Kill{\alpha}{\beta} &\defeq (\widehat\alpha\wedge\widehat\beta)\otimes\Kill{A}{B}\end{align*}
We will use the following facts freely. 
  If $\alpha$ and $\beta$ are respectively of degree $p$ and $q$, and
  if $\xi\in \chi^\infty(M)$
  \begin{align}
    \alpha\ww\beta&=(-1)^{pq+1}\beta\ww\alpha\\
    i_\xi(\alpha\ww\beta)&=i_\xi\alpha\ww\beta +(-1)^p \alpha\ww i_\xi\beta\\
  \Kill{\gamma}{\beta\ww\alpha}&=(-1)^{pq+1}\Kill{\gamma\ww\alpha}{\beta}.\label{killexchange}
  \end{align}
  If $\alpha$ and $\beta$ are 1-form and $\gamma$ a $0$-form then
    \begin{align}
     \alpha\ww(\gamma\ww\beta)+\beta\ww(\gamma\ww\alpha)=\gamma\ww(\alpha\ww\beta). \label{Jacobi} \end{align}
  If furthermore $\mathcal G$ is equipped with a connection and if
  $\d$ is the corresponding exterior derivative on
  $\Omega^*(M,\mathcal G)$, then for $\alpha$ of degree $p$
  \begin{align}
    \d(\alpha\ww\beta)&=\d\alpha\ww\beta+(-1)^p\alpha\ww\d\beta.
  \end{align}

Finally recall a convention of notation: if $\alpha$ belongs to $\Omega(\X,\mc G)$, and $f$ is a map from $\Sigma$ to $\X$, then $f^*(\alpha)$ is a form with values in $f^*(\mc G)$.

Let $S$ be a Riemann surface, $\mc G$ a Lie algebra bundle over $S$ equipped with a section of Cartan involution $\rho$. Then we have

\begin{proposition}\label{Sign}
Let $\alpha$ be a $(1,0)$ form with value in $\mc G$, then
\begin{align}
i\cwdot\int_S \Kill{\alpha}{\rho(\alpha)}\leq 0,
\end{align}
with equality if and only if $\alpha=0$.
Conversely if $\alpha$ is of type $(0,1)$ then 
\begin{align}
i\cwdot\int_S \Kill{\alpha}{\rho(\alpha)}\geq 0,
\end{align}
with equality if and only if $\alpha=0$.
\end{proposition}
\begin{proof}
Let us first consider a form $\alpha=\widehat\alpha\otimes A$, where $A$ is a section of $\mc G$ and $\widehat\alpha$ a $(1,0)$-form on $S$. Then
$\rho(\alpha)=\overline{\widehat\alpha}\otimes \rho(A)$.
Thus for every $x$ in $S$, $u\in\T S$, we have 
\begin{align}
i\cwdot \Kill{\alpha}{\rho(\alpha)}_x(u,Ju)=i\cwdot  \Kill{A}{\rho(A)}_x\widehat\alpha\wedge\overline{\widehat\alpha}_x\cwdot(u,Ju)\leq 0.
\end{align}
with equality only if $\alpha_x=0$. Given a point $x$, we can now decompose any form $\alpha$ as the sum of decomposable forms $\widehat\alpha_i\otimes A_i$ where the $A_i$ are pairwise orthogonal with respect to $A,B\mapsto\Kill{A}{\rho(B)}$ and obtain the same statement. The result follows.

\end{proof}

\subsection{Geometry of the space of Hitchin--Kostant quadruples}

\subsubsection{Vector subbundles}
The group $\ms G$ act by conjugation on the space $\X$ of Hitchin--Kostant quadruples. Moreover once we fix a
Hitchin--Kostant quadruple $(\mk h,\dep,\rho,\lambda)$ then $\X $ is identified with
$\ms G/\ms T$ by Proposition \ref{X}, where $\ms T\subset \ms H$ be
the torus fixed by the involutions $\lambda$ and $\rho$. Recall then
that $\ms T$ is compact and is the maximal torus of the maximal
compact of $\ms G_0$ by Proposition \ref{HitchinTorus}.

Let $\mc G$ be the trivial bundle $\mathcal G\defeq \mk g\times \X $
equipped with the trivial connection $\D$. The following definitions introduce some of the geometry of $\X$.
\begin{definition}{\sc[Vector bundles]}\label{def:subh}
  We denote by $\mc H$, $\mc T$, $\mc H_0$, the
  subbundles of $\mc G$ whose fibre at $(\mk h,\Delta^+, \rho,\lambda)$ are
  respectively
  \begin{align}
    \mk h&,\\ \mk t&\defeq\{u\in\mk h\mid \sigma(u)=u,\rho(u)=u\},\\
    \mk h_0&\defeq\{u\in\mk h\mid \forall v\in\mk t,
    \Kill{u}{v}=0\}.\end{align}
Using the root system, we also have a decomposition:
  \begin{equation}
    \mc G=\mc H\oplus\bigoplus_{\alpha\in\Delta} \mc G_\alpha,\label{fondec2}
  \end{equation}
  such that at a point $x=(\mk h,\dep,\rho,\lambda)$,  $\mc H_x=\mk
  h$, and $\left(\mc G_\alpha\right)_x=\mk g_\alpha$, where $\mk
  g_\alpha$ are the root spaces
  associated to $\mk h$.
\end{definition}

At any point $x=(\mk
  h,\dep,\rho,\lambda)$ of $\X $, we identify $\T_x\X $ with $\mk g/\mk
  t$. We denote by $E^\perp$ the orthogonal with respect to the Killing form.

\begin{definition}{\sc[Maurer--Cartan Form]}\label{def:MC} The {\em
    Maurer--Cartan} form on $\X $ is the form $\omega\in\Omega^1(\X,\mc T^\perp)\subset 
  \Omega^1(\X,\mc G)$ defined as the inverse of the projection from $\mk t^\perp$ to $\mk g/\mk t=\T_x\X $.
\end{definition}
In the sequel we will sloppily identify $\omega$ with $\ad(\omega)$.
Observe that the Maurer--Cartan form satisfies
\begin{equation}
  \forall u\in\mk t,\ \  \Kill{u}{\omega}=0.  \label{MC}
\end{equation}

\subsubsection{Connections}
Let $b:\ms G\to \operatorname{End}(V)$, be a linear representation. Let $\mc V=V\times\ms{X}$ be the associated trivial bundle over $\ms{X}$ equipped with the trivial connection $\D$. We say a section $\sigma$ of $\mc V$ is 
{\em $\ms G$-equivariant} if
 $$\sigma(g\cdot x)=b(g)\cdot\sigma(x).$$
In particular, since $\sigma(g\cdot x)=b(g)\cdot\sigma(x)$, then for every $u\in\mk g$,
$$
\D_{\tilde{u}}\sigma=\T b(u)\cdot\sigma,
$$
where $\tilde{u}$ is the vector field on $\X$ associated to $u$ and $\T b$ is the tangent map to $b$.
Thus a $\ms G$-equivariant section of $\mc V$ is parallel under the connection
$$
\nabla\defeq\D-\T b\circ\omega.
$$
In the next proposition, we will use  this fact when $b$ is the adjoint representation.

\begin{proposition}\label{vec-bun-des}
Let 
$$
\nabla:=\D-\omega,
$$ 
 where $\omega$ is the Maurer--Cartan
  form.
Then
\begin{enumerate}
	\item The subbundles $\mc H$, $\mc T$ and $\mc H_0$ are parallel for $\nabla$,
	\item The Cartan involution  $\rho$ is parrallel for $\nabla$.
		\item The curvature $R^\nabla$ belongs to $\Omega^2(\X,\mc T)$,
\end{enumerate}
\end{proposition}
In particular since $\nabla+\omega$ is flat, the curvature of $\nabla$ is given by the equation
\begin{equation}
  \d^\nabla\omega+\omega\ww\omega +R^\nabla=0.\label{curvfund}
\end{equation}
The second assertion implies that the metric $g(u,v)=-\Kill{\rho(u)}{v}$ is parrallel for $\nabla$ or, in other words, that $\nabla$ is metric.

\begin{proof} Let $S$ be an $\ms G$-equivariant section of $\operatorname{End}(\mc  G)$. Thus, $S$ is parallel under the connection $\nabla^0$ on $\operatorname{End}({\mk g})$ such that
$$
\nabla_u^0S=\D_u S - [\ad(\omega(u)),S].
$$
A short computation yields that 
\begin{align*}
0=(\nabla^0_u S)(v)&=\nabla_u(S(v))-S(\nabla_u v).
\end{align*}
Thus if $w=S(v)$ is section of $\operatorname{Im}(S)$ then, \begin{align*}
\nabla_u w=S(\nabla_u v)\in \operatorname{Im}(S).
\end{align*}
In other words, $\operatorname{Im}(S)$ is parallel under $\nabla$. 

We obtain the first part of the result by applying this observation for the Killing orthogonal projections on $\mathcal H$, $\mathcal T$ and $\mathcal H_0$ as well as well as the projection on $\mc G_\alpha$ from the root space space decomposition.
Observe that this orthogonal projection exists since the restriction of the Killing form is non degenerate, since the corresponding Lie subalgebras are reductive. 

Applying this very same observation to $\rho$, we get the second assertion.

Since $\rho$ and $\lambda$ are parallel, it follows that the Lie algebra  $\mk k$
(associated to the maximal compact of $\ms G$) and $\mk g_0$
(associated to the split real form of $\ms G$) are both parallel.
These two algebras being both self normalizing, it follows that
$R^\nabla\in\Omega^2(\X ,\mk g_0\cap\mk k)$. Similarly since $\mk h $
is parallel and self normalizing, we further have that
$$R^\nabla\in\Omega^2(\X ,\mk g_0\cap\mk k\cap \mk h).$$
The last statement of the proposition follows from the fact that
 $$
 \mk g_0\cap\mk k\cap \mk h=\mk t.\qedhere
 $$
\end{proof}

Conversely, we have

\begin{proposition}\label{VecBundMan}
  Let $\widehat{\mc G}$  be a Lie algebra bundle over a  simply connected manifold $M$ equipped with
  \begin{enumerate}
  \item a smoothly varying Hitchin--Kostant quadruple $\kappa:m\mapsto (\widehat{\mk
      h}_m,\widehat\Delta^+_m,\widehat{\rho}_m,\widehat{\lambda}_m)$ in every fibre.
  \item a connection $\widehat\nabla$ for which the Hitchin--Kostant quadruple $\kappa$ is
    parallel.
  \item an element $\widehat\omega\in\Omega^1(M,\mc G)$, such that
    $\widehat \D\defeq \widehat\nabla+\ad(\widehat\omega)$ is flat and
    moreover

\begin{equation}
  \forall u\in\widehat{\mk t},\ \ \braket{u|\widehat\omega}=0,\label{womt}
\end{equation}
where $\widehat{\mk t}\defeq\{u\in\widehat{h}\mid
\widehat\rho(u)=u=\widehat\lambda(u)\}$.
\end{enumerate}

Then there exists a map $f$ from $M$ to $\X $, unique up to post composition by an element of $\ms G$, such that $\widehat{\mc
  G}$, $\widehat{\mc H}$, $\widehat\Delta^+$, $\widehat{\rho}$, $\widehat{\lambda}$,
$\widehat\omega$,$\widehat\nabla$ and $\widehat{\sigma}$ are the
pulled back of the corresponding objects in $\mc
G$.  \end{proposition}

As an immediate corollary, we get

\begin{corollary}\label{VecBundManCor}
   Let $\widehat{\mc G}$  be a Lie algebra bundle over a connected manifold $M$ equipped with the same structure as in Proposition \ref{VecBundMan}, 
then there exists 
\begin{enumerate}
\item a representation $\rho$ of $\pi_1(M)$ in $\ms G$ unique up to conjugation,
\item a $\rho$-equivariant map $f$ from the universal cover $\widetilde{M}$ of $M$,  in $\X $ satisfying the properties in the conclusion of Proposition \ref{VecBundMan}.
\end{enumerate}
  \end{corollary}

\noindent{\em Proof of Proposition \ref{VecBundMan}.}  Since $\widehat\D$ is flat, we may as well
  assume that $\widehat{\mc G}$ is the trivial flat bundle $\mc G=\mk
  g\times M$. Thus the map $f:m\mapsto (\widehat{\mk
    h}_m,\widehat\Delta^+_m,\widehat{\rho}_m,\widehat{\lambda}_m)$ is now a map from $M$
  to $\ms X$. By construction, $\widehat{\mc G}$, $\widehat{\mc H}$, $\widehat\Delta^+$, 
  $\widehat{\rho}$, $\widehat\lambda$ and $\widehat{\D}$ are the pullbacks by $f$ of $\mc G$,
  $\mk h$, $\dep$, $\rho$, $\lambda$ and $\D$.  Thus
  $(\widehat{h},\widehat\Delta^+,\widehat{\rho},\widehat{\lambda})$ is parallel both
  for $\widehat{\nabla}$ and $f^*\nabla$. Since the stabiliser in
  $\mk g$ of $(\mk h,\dep,\rho,\lambda)$ is $\mk t$,
$$
f^*(\omega)-\widehat{\omega}=\widehat{\nabla}-f^*\nabla\in\Omega^1(M,\mk
t).
$$
However by Hypothesis \eqref{womt} and Equation \eqref{MC},
$$
f^*(\omega)-\widehat{\omega}\in\Omega^1(M,\mk t^\perp).
$$
Thus $f^*(\omega)=\widehat{\omega}$\,. Then $f^*\nabla=\widehat\nabla$
and the proof of the second assertion of the proposition is completed. \qedhere

\subsubsection{The real structure on the space of Hitchin--Kostant quadruples}

Since $\ms T$ is a subgroup of $\ms G_0$, it follows that each leaf of
the foliation by right $\ms G_0$-orbits of $\ms G$ is invariant by the
action of $\ms T$, thus giving rise to a foliation $\bf F$ of $\ms
G/\ms T$ whose leaves are all isomorphic to $\ms G_0/\ms T$. Since
this foliation is left invariant by the action of $\ms G$, it gives a
foliation, that we also denote $\bf F$, on $\X $. Since $\lambda$ preserves $\mk t$ and thus $\mk t^\perp$, we obtain a
{\em real structure} $v\to\CG{v}$ on $\X $ by setting
$$
\omega(\overline{v})=\lambda(\omega(v)).
$$
One then immediately have
\begin{proposition}\label{real0}
  The tangent distribution $\T\bf F$ of the foliation $\bf F$ is given
  by
$$
\T {\bf F}\defeq \{u\in \T\X \mid \CG{u}=u\}.
$$
\end{proposition}

\begin{proof}
  Indeed, $\mk g_0$ is the set of fixed points of $\lambda$ in $\mk
  g$.
\end{proof}
\subsubsection{The space of Hitchin--Kostant quadruples and the symmetric space}
The map $p:(\mk
h,\dep,\rho,\lambda)\mapsto\rho$ defines a natural $\ms G$ equivariant
projection $p$ from the space $\X$ of Hitchin--Kostant quadruples to the symmetric space $\ms S(\ms G)$. The fibres of this
projection are described as follows. Since $\ms T$ is a subgroup of
the maximal compact $\ms K$, each leaf of the foliation by the right
$\ms K$-orbits on $\ms G$ is invariant by $\ms T$, thus giving rise to a
foliation on $\ms G/\ms T$. This foliation is invariant under the left
$\ms G$-action and thus gives a foliation $\bf K$ on $\X $. The leaves
of $\bf K$ are precisely the fibres of the projections $p$ from $\X
$ to $\ms S(\ms G)$.

We should remark that the construction of the  Maurer--Cartan form 
holds any homogeneous $\ms G$-space with a reductive stabiliser, in particular for the
symmetric space $\ms S(\ms G)$ as we already did at the end of Section \ref{realcond}. In particular,  one gets
\begin{proposition} \label{MaurerSym} The canonical $\mk g$-bundle $\mc G$
  over $\ms S(\ms G)$ identifies with $\T_{\mathbb C}\ms S(\ms
  G)$. Let then $i$ be the associated  injection of 
  $\T \ms S(G)\to \mc  G$. Then if $\alpha$ is the identity map of $\T\ms S(\ms G) $
seen as an element of $\Omega^1\left(\ms S(\ms G) ,\T\ms
    S(\ms G) \right)$, then
  \begin{itemize}
  \item $i\circ \alpha$ is the Maurer--Cartan form of $\ms S(\ms
    G)$,
  \item Moreover $p^*\left(i\circ \alpha\right)=\frac{1}{2}(\omega+\rho(\omega))$.
  \end{itemize}

\end{proposition}

\begin{proof}
Let us just prove the second assertion. In general, if $\ms H$ and $\ms L$ are reductive subgroups of a semi simple group $\ms G$ with $\ms  H\subset \ms L$, $\omega_H$ and $\omega_L$ the respective Maurer--Cartan forms on $\ms G/\ms H$, $\ms G/\ms L$, if $p$ is the projection $\ms G/\ms H\to \ms G/\ms L$ and $\pi$ the projection from $\mk g^\perp$ to $\mk l^\perp$, then $p^*(\omega_L)=\pi\circ\omega_H$. This proves the second assertion. \end{proof}

\subsection{The cyclic decomposition of the Maurer-Cartan form}
Let $\omega$ be the Maurer--Cartan on $\X$ as in Definition
\ref{def:MC} with value in the bundle $\mc G$. We use the
decomposition \eqref{fondec2} to write
\begin{equation}
  \omega=\omega_0+\sum_{\alpha\in\Delta}\omega_\alpha, \hbox{ with } 
  \omega_\alpha\in\Omega^1(\X,\mc G_\alpha),\   \omega_0\in\Omega^1(\X,\mc H)\label{MC-dec}
\end{equation}
Actually, one has by Equation \eqref{MC} that
\begin{align}
  \omega_0\in\Omega^1(\X,\mc H_0).\label{eq:MCcyc0}
\end{align}
From Equation \eqref{curvfund} and since $R^\nabla$ takes values in $\mk t$, we obtain that for all $\alpha\not=0$,
\begin{equation}
  -\d^\nabla\omega_\alpha=2\omega_0\ww\omega_\alpha+\mathop{\sum_{\beta,\gamma \in \Delta}}_{\beta+\gamma=\alpha}\omega_\beta \ww\omega_\gamma.
\end{equation}

We consider the following projections (whose pairwise products are
zero) coming from the projection on the Lie algebra defined in
equations \eqref{pidec0} and that we denote by the same symbol by a
slight abuse of notations.
\begin{align}
  \pi_0: \mathcal G&\to \mc H,\cr \pi:\mc G&\to\mc G_Z\defeq
  \bigoplus_{\alpha\in Z}\mc G_\alpha,\cr \pi^\dagger:\mc G&\to \mc
  G_{Z^\dagger}\defeq \bigoplus_{\alpha\in Z^\dagger}\mc G_\alpha,\cr
  \pi_1:\mc G&\to \mc G_1\defeq \bigoplus_{\alpha\not\in Z\cup
    Z^\dagger}\mc G_\alpha.\label{def:subbun}
\end{align}
Observe that
$$
\pi+\pi^\dagger+\pi_0+\pi_1=1.
$$
Obviously, Proposition \ref{crucial} extends word for word for the
various brackets of the vector subbundles described in the equations
\eqref{def:subbun}.

\begin{definition}
  The {\em cyclic decomposition} of the Maurer--Cartan form $\omega$
  is
  \begin{equation}
    \omega=\omega_0+\omega_1+\phi+\phi^\dagger,\label{eq:MCcyc}
  \end{equation}
  where $\omega_0=\pi_0(\omega)$, $\omega_1=\pi_1(\omega)$,
  $\phi=\pi(\omega)$ and $\phi^\dagger=\pi^\dagger(\omega).$
\end{definition}

We will use the following
\begin{proposition}\label{cyclic-dec}
  Let $\omega$ be any 1-form on $\ms X$ with values in $\mc G$, then
  \begin{align}
    \pi_0(\omega\ww\omega)={}&2\cwdot \phi\ww\phi^\dagger+
    \pi_0(\omega_1\ww \omega_1),\label{zeta090}
    \\
    \pi_1(\omega\ww\omega)={}&2\cwdot\pi_1(\omega_0\ww\omega_1+\omega_1\ww\phi+\omega_1\ww\phi^\dagger)\cr
    &+\pi_1(\omega_1\ww\omega_1 +\phi\ww\phi+\phi^\dagger
    \ww\phi^\dagger).\label{pi1om}\\
    \pi(\omega\ww\omega)={}&2\cwdot \omega_0\ww\phi+2\cwdot\pi(\omega_1\ww\phi^\dagger)\cr
       &+\pi(\phi^\dagger\ww\phi^\dagger)+\pi(\omega_1\ww\omega_1)\label{piom}.      
       \end{align}
\end{proposition}
\begin{proof}
Let us consider the cyclic decomposition 
\begin{align}
  \omega=\omega_0+\omega_1+\phi+\phi^\dagger.
\end{align}
Then 
\begin{align}
  \omega\ww\omega&=\omega_0\ww\omega_0+\omega_1\ww\omega_1+\phi\ww\phi+\phi^\dagger\ww\phi^\dagger\cr
  &+2\omega_0\ww\omega_1+2\omega_0\ww\phi+2\omega_0\ww\phi^\dagger\cr
  &+2\omega_1\ww\phi+2\omega_1\ww\phi^\dagger\cr
  &+2\phi\ww\phi^\dagger.
\end{align}
According to Proposition \ref{crucial}, we have
  \begin{align*}
    \pi_0(\phi\ww\phi)&=\pi_0(\phi^\dagger\ww\phi^\dagger)=0,\\
    \pi_0(\phi\ww \omega_1)&=\pi_0(\phi^\dagger\ww \omega_1)=0,\\
    \pi_0(\phi\ww\phi^\dagger)&=\phi\ww\phi^\dagger.
  \end{align*}
  Thus, using the fact that $\mk h$ is commutative, and normalises
  $\mc G_1$, $\mc G_Z$ and $\mc G_{Z^\dagger}$ we get Equation
  \eqref{zeta090}.
We use again Proposition \ref{crucial} to get that
  \begin{align}
    \pi_1(\omega_0\ww\omega_0)=0,&\ \ \ \ \pi_1(\omega_0\ww\phi)=0,\cr
    \pi_1(\phi\ww\phi^\dagger)=0,& \ \ \ \
    \pi_1(\omega_0\ww\phi^\dagger)=0.\label{ifi12}
  \end{align}
  Thus
  \begin{align}
    \pi_1(\omega\ww\omega)&=2\cwdot\pi_1(\omega_0\ww\omega_1+\omega_1\ww\phi+\omega_1\ww\phi^\dagger)\cr
    &+\pi_1(\omega_1\ww\omega_1 +\phi\ww\phi+\phi^\dagger
    \ww\phi^\dagger).
      \end{align}
Then finally, by  Proposition \ref{crucial} we have that
\begin{align}
  \pi(\omega_0\ww\omega_0)=0,\ \ \pi(\omega_0\ww\omega_1)=0,\cr
  \pi(\omega_0\ww\phi^\dagger)=0,\ \ \pi(\phi\ww\phi)=0,\cr
  \pi(\omega_1\ww\phi)=0, \ \pi(\phi\ww\phi^\dagger)=0.
\end{align}
Thus
\begin{align}
  \pi(\omega\ww\omega)=2\pi(\omega_0\ww\phi)+2\pi(\omega_1\ww\phi^\dagger)+\pi(\omega_1\ww\omega_1)+\pi(\phi^\dagger\ww\phi^\dagger)
\end{align}
The proof of the proposition is completed.
\end{proof}
\section{Higgs bundles and Hitchin theory}
In this section, we will recall the definition of a Higgs bundle and
sketch some of Hitchin theory. Higgs bundles have been studied extensively by many authors. The special case of $\ms{Sp}(4,\mathbb R)$ has been in particular studied by Bradlow, Garc{\'\i}a-Prada, Gothen and Mundet i Riera in \cite{Garcia-Prada:2004}, \cite{Bradlow:2012dl} and $\ms{Sp}(2n,\mathbb R)$ in \cite{GarciaPrada:2013be}.

\subsection{Higgs bundles and the self duality equations}

We recall some definition and results from Hitchin
\cite{Hitchin:1987}. 
We recall that a {\em Higgs (adjoint) bundle} over a Riemann surface $\Sigma$,
is a pair $E=(\widehat{\mc G},\Phi)$ where
\begin{enumerate}
\item $\widehat{\mc G}$ is a holomorphic Lie algebra bundle over
  $\Sigma$,
\item The {\em Higgs field} $\Phi$ is a holomorphic section of $\widehat{\mc G}\otimes \mc K$, where $\mc K$ is the canonical bundle of
  $\Sigma$.
\end{enumerate}


Let $\nabla$ be a connection on $\widehat{\mc G}$ compatible with the
holomorphic structure and $\widehat\rho$ a section of the bundle of antilinear 
automorphisms of $\widehat{\mc G}$ such that the restriction to every
fibre is a Cartan involution with respect to a maximal compact. Let
$\Phi^*=-\widehat\rho(\Phi)$, and $R^\nabla$ the curvature of
$\nabla$. We say that $(\nabla,\widehat\rho)$ is a {\em solution the
  self duality equations} if
\begin{align}\label{def:sd}
  \nabla\widehat\rho&=0,\cr \d^\nabla \Phi&=0,\, \, \, 
  \d^\nabla\Phi^*=0
  \cr R^\nabla&=2\cwdot\Phi\ww\Phi^*.
\end{align}
The last three equations are equivalent to the fact that
$\nabla+\Phi+\Phi^*$ is flat and the curvature of $\nabla$ is of type
$(1,1)$. Observe also that $\widehat\rho$ and the holomorphic structure totally determines $\nabla$:
by the first and last  equation in \eqref{def:sd},  $\nabla$ is the Chern connection of the Hermitian bundle $(\widehat{\mc G},
\widehat\rho)$.

Given a Higgs bundle $E=(\widehat{\mc G},\Phi)$ for which there exists a solution of the self duality equations, over a
closed Riemann surface $\Sigma$, let $(\nabla,\widehat\rho)$ be the solution
of the self duality equations. The
{\em representation associated to the Higgs bundle $E$} is the
monodromy of the flat connection $\nabla+\Phi+\Phi^*$.

The {\em Hopf differential} of the Higgs bundle is the quadratic
holomorphic differential $\Kill{\Phi}{\Phi}$ where $\Kill{\cdot}{\cdot}$ denotes
the Killing form. From \cite{Donaldson:1987}, solutions of the self
duality equation are interpreted as equivariant harmonic mappings. Those harmonic mappings for which
the  Hopf differential vanishes are conformal harmonic mappings,
or in other words branched minimal immersions \cite{Gulliver:1973}.

\subsection{The Hitchin section}\label{HitchSec}
Let us now recall the construction by Hitchin \cite{Hitchin:1992es} of Higgs bundles from holomorphic differentials, using the notation of
our preliminary paragraph. Let $\Sigma$ be a closed surface. Given a
complex Lie group $\ms G$. We choose a Cartan subalgebra $\mk h$ and
an $\mk h$-principal Lie algebra $\mk s$ generated by $(X,a,Y)$ as in Section
\ref{sec:Principal}. Let $m_1,\ldots, m_{\ell}$ be the exponents of $\ms G$, so that we have the decomposition of $\mk g$ into irreducible representations of $\mk s$ as $\mk g=\bigoplus_{i=1}^{\ell}\mk v_i$, with $\dim(\mk v_i)=2m_i+1$. Let $e_i$ be
an element of $\mk v_i$ of highest weight with respect to the
action of the principal Lie algebra generated by $(X,a,Y)$.

Let us also write the decomposition under the grading by the element
$a$ as
\begin{align}
  \mk g&\defeq \bigoplus_{i=-m_{\ell}}^{i=m_{\ell}}\mk g^{(i)},
\end{align}
where $\mk g^{(m)}\defeq \{u\in \mk g\mid [a,u]=m\cdot u\}$.  Observe that
$e_i\in\mk g_{m_i}$ and
\begin{align}
  Y\in\mk g^{(-1)}&=\bigoplus_{\alpha\in\Pi}\mk g_{-\alpha}.\label{ying}
\end{align}
Moreover $\mk g^{(0)}=\mk h$ is the centraliser of $a$. Let us now
consider the Lie algebra bundle
\begin{align}
  \widehat{\mc G}&\defeq
  \bigoplus_{i=-m_{\ell}}^{i=m_{\ell}}\widehat{\mc
    G}^{(m)},\label{decompE}
\end{align}
where $\widehat{\mc G}^{(m)}\defeq \mk g^{(m)}\otimes \mc K^m$ and $\mc K$ is the
canonical bundle of $\Sigma$.  We write $\widehat{\mc
  H}{\eqdef}\widehat{\mc
  G^{(0)}}=\mk h \otimes \mc K^0$. The fibre of $\widehat{\mc H}$ is a
Cartan subalgebra equipped with a choice of positive roots (given by
the element $a$).  We then denote by
$$
\widehat{\mc G}\defeq \widehat{\mc H} \oplus \bigoplus_{\alpha\in \Delta} \widehat{\mc
  G}_\alpha,
$$
the corresponding root space decomposition where $\widehat{\mc
  G}_\alpha$ is the eigenspace associated to the root $\alpha$.

The {\em Hitchin section} then associates to a family of holomorphic
differentials ${\rm q}\defeq (q_1,\ldots, q_{\ell})$ where $q_i$ is of
degree $m_i+1$ the Higgs bundle $H({\rm q})\defeq (\widehat{\mc G},\Phi_q)$ where
\begin{equation}
  \Phi_{\rm q}\defeq  Y+\sum_{i=1}^\ell e_i\otimes q_i\in H^0(\Sigma,
  \widehat{\mathcal G}\otimes \mc K).\label{def:phiq}
\end{equation}
Observe that $[\mk g^{(m)},\mk g^{(m')}]\subset \mk g^{(m+m')}$ and thus the Lie algebra structure on 
$\widehat{\mc G}$ is well defined. 

By Section 5 of \cite{Hitchin:1992es} based on Theorem 7 of \cite{Kostant:1963vp}, we have
\begin{proposition}\label{InvPoly}
There exist homogeneous invariant polynomials $p_i$ on $\mk g$ of degree $m_i+1$ such that
$
  p_i(\Phi_{\rm q})=q_i
$.
 
\end{proposition}

Observe also that \begin{equation}
  \widehat\sigma(\Phi_{\rm q})=-\Phi_{\rm q} \label{eq:PhiSigma}
\end{equation}
where $\widehat\sigma$ is the unique involution (holomorphic) associated to $\mk s$ by
Proposition \ref{UniqueSigma}. Hitchin then proved in
\cite{Hitchin:1992es}.

\begin{theorem}{\sc [Hitchin]}\label{th:H2}
  The self duality equations associated to the Higgs bundle $\mc H({\rm q})$ admit a unique solution  $(\nabla,\widehat\rho)$. Moreover $\nabla\widehat\sigma=0$. In particular the monodromy is
  with values in $\ms G_0$. Finally if $\rm q=0$, then the monodromy is with values in the principal $\ms{SL}_2$ and is the uniformisation of the underlying Riemann surface.
 
\end{theorem}
The first assertion uses more general results of Hitchin and Simpson \cite{Hitchin:1987}, \cite{Simpson:1988ch}.
The second assertion follows at once from the uniqueness of the
solutions of the self duality equations. Also since $\Phi({\rm q})$ is injective it immediately follows that the harmonic mappings associated to the Higgs bundles $\mc H({\rm q})$ are immersions (See \cite{Sanders:2014vb} for details).

\subsubsection{Hitchin component}\label{sec:HitchinRep}
Finally, Hitchin proved

\begin{theorem}{\sc [Hitchin]}\label{th:H3}
  Given a Riemann surface $\Sigma$, the map which associates to ${\rm
    q}$ the monodromy associated to the Higgs bundle $H({\rm q})$ is a
  parametrisation of a connected component $\mc H(\Sigma,\ms G_0)$ of the character variety of
  representations from $\pi_1(\Sigma)$ to $\ms G_0$.
\end{theorem}
The case of $\ms G_0=\ms{SL}(2,\mathbb R)$ had been treated independently by M. Wolf in \cite{Wolf:1989uk}, where he directly considered the quadratic differential as the Hopf differential of the harmonic mapping. The connected component $\mc
H(\Sigma,\ms G_0)$ is now
called the {\em Hitchin component} and its elements 
are {\em Hitchin representations}. The {\em Fuchsian locus}
is the subset of {\em Fuchsian representations}, that is 
those representations that are discrete faithful and with values in a
principal $\ms{SL}_2$. By \cite{Labourie:2006} and \cite{Fock:2006a} Hitchin
representations are discrete faithful.

\subsection{Cyclic Higgs bundles}\label{HichSec}
A {\em cyclic Higgs bundle} is by definition the image $H({\rm q})$ by
the Hitchin section of a family of holomorphic differential ${\rm
  q}\defeq (q_1,\ldots, q_{\ell})$ where $q_i=0$ when
$i\not=\ell$. The corresponding Higgs Field $\Phi_{\rm q}$ is called
{\em cyclic}.  Cyclic Higgs bundles were studied by Baraglia in
\cite{Baraglia:2010vi} in relation with the affine Toda lattice.
 
It follows immediately from the construction that

\begin{proposition}\label{HiggsCyclic}
  For a cyclic Higgs field $\Phi_{\rm q}$, we have
  \begin{equation}
    \Phi_{\rm q}\in\Omega^1\left(\Sigma,\bigoplus_{\alpha\in Z}\widehat{\mathcal G}_\alpha\right).
  \end{equation}
\end{proposition}

\begin{proof} By Proposition \ref{etaHiggs},
  $e_\ell\in\mk{g}_\eta$. By Equation \eqref{ying}
$$
Y\in\widehat{\mc  G}^{(-1)}\otimes\mc K=\bigoplus_{\alpha\in\Pi}\left(\widehat{\mathcal
  G}_{-\alpha}\otimes\mc K\right).
$$
The proposition follows.
\end{proof}

The following is implicit in Baraglia's paper \cite{Baraglia:2010vi}
but not stated as such. We write the proof using arguments borrowed
from this article. Similar results are found in \cite{Collier:2014tn}.

\begin{proposition}\label{hpara}{\sc [Baraglia]} Let
  $(\nabla,\widehat\rho)$ be the solution of the self duality equation
  associated to a cyclic Higgs bundle $(\widehat{\mc G},\Phi_{\rm q})$. Then $\widehat{\mc H}$ is
  parallel under $\nabla$ and globally invariant by $\widehat\rho$.
\end{proposition}

\begin{proof} Let $\zeta$ be a primitive 
  $({m_{\ell}+1})$-th root of unity. Let $\psi$ be the section of the bundle of automorphisms of $\widehat{\mc G}$
  whose restriction to $\widehat{\mc G}^{(m)}$ is the multiplication by
  $\zeta^m$. Observe now that
$$
\psi(\Phi_{\rm q})=\zeta^{-1}\Phi_{\rm q}.
$$
It follows that $(\nabla,\widehat\rho)$, being also a solution of
Hitchin equation for $(E,\zeta^{-1}\cwdot\Phi_{\rm q})$ is then a
solution for $(E,\psi(\Phi_{\rm q}))$. It follows that
$$
\psi^*\nabla=\nabla,
$$
and thus $\psi$ is parallel for $\nabla$ and, in particular, so is
$\widehat{\mc H}$ which is the centralizer of $\psi$. Similarly, we
obtain that $\psi$ anticommutes with $\widehat\rho$. Thus  $\widehat{\mc H}$ which is the centralizer of $\psi$  is
globally preserved by $\widehat\rho$.
\end{proof}


\subsection{Harmonic mappings and Higgs bundles}
 We recall the following facts relating equivariant harmonic mappings and Higgs bundles (see \cite{Donaldson:1987} in the case of $\ms{SL}(2,\mathbb C)$). Let as usual $\ms G$ be a complex Lie group, and the bundle $\mc G$, form $\omega\in\Omega^1(\ms{S(G)},\mc G)$ and $\rho\in\Gamma\left(\operatorname{Aut}(\mc G)\right)$ as defined in paragraph \ref{geom-sym}.

 Recall that given a Riemannian manifold $M$ and a representation $\delta$ of $\pi_1(M)$ in $\ms G$, a {\em $\rho$-equivariant map} $F$  from $M$ to $\ms S(\ms G)$, is a map $\tilde F$ from 
the universal cover $\tilde M$ of $M$, with values in $\ms S(\ms G)$ so that $\tilde F(\gamma\cdot x)=\rho(\gamma)\cdot \tilde F(x)$, where $\gamma\in\pi_1(M)$. Then the energy density of $\tilde F$ is $\pi_1(M)$-invariant and thus give rises to a function $e(F)$ on $M$ also called the {\em energy density} of $F$.

 \begin{theorem}
 Let $f$ be an harmonic mapping from a Riemann surface $S$ in $\ms{S(G)}$, then $(f^*(\mc G), (f^*\omega)^{(1,0)})$ is a Higgs bundle. Moreover $f^*(\rho)$  satisfies the self duality equations. Conversely, let $S$ be a simply connected Riemann surface, $(E,\Phi)$ a Higgs bundle and $\rho_E$ a solution of the self-duality equations. Then there exists an harmonic mapping $f$ from $S$ to $\ms{S(G)}$ unique up to the action of $\ms G$ so that 
 $$
 (E,\Phi,\rho_E)=(f^*\mc G, (f^*\omega)^{(1,0)}, f^*\rho).
 $$
 \end{theorem}
Moreover using the notation of paragraph \ref{prel-form}, we have

\begin{proposition}\label{ener} Let $S$ be a Riemannian surface equipped with the area form $\d \mu$.
Let $f$ from $S$ to $\ms{S(\ms G})$ be an equivariant harmonic mapping. Let $e(f)$ be the energy density on $S$,
then
$$
\operatorname{Energy}(f)\defeq\frac{1}{2}\int_S e(f)\,\d\mu=i\cwdot \int_S 
\Kill
{(f^*\omega)^{(1,0)}}
{(f^*\omega)^{(0,1)}}.$$ 
\end{proposition}
\begin{proof}
We have that
$$
e(f)\,\d\mu=
-\Kill{f^*\omega}{f^*\omega\circ J}
= 2i\cwdot 
\Kill
{(f^*\omega)^{(1,0)}}
{(f^*\omega)^{(0,1)}}.
$$
\end{proof}

\section{Cyclic surfaces}\label{sec:CS}

We interpret the Cartan involution $\rho$ as the real structure on $\mathfrak g$
coming from the complexification of $\mk k$. Let also $\sigma$ be the
involution constructed in Section \ref{realsplit}. Let finally
$\lambda=\sigma\circ\rho$ be the real structure on
$\mk g$ coming from the complexification of $\mk g_0$. We will use in
this section the decomposition \eqref{eq:MCcyc}.

\begin{definition}{\sc[cyclic maps]}\label{cyc-map}
  A map $f$ from a surface $\Sigma$ to $\X $ is {\em cyclic}
  if \begin{enumerate}
  \item $f^*(\omega_1)=0$,\label{def:cyc1}
  \item $f^*(\omega_0)=0$,\label{def:cyc2}
  \item $f^*(\phi\ww\phi)=0$, \label{def:cyc5}
  \item
    $f^*\left(\rho(\phi)\right)=-f^*({\phi^\dagger})$,\label{def:cyc3}
  \item
    $f^*\left(\lambda(\omega)\right)=f^*\left(\omega\right)$.\label{def:cyc4}
  \item if $\beta$ is a simple root,
    $f^*(\omega_\beta)$ never vanishes,\label{def:cyc6}
  \end{enumerate}

\end{definition}

The notion of cyclic surfaces is cousin to that of $\tau$-maps studied
in \cite{Bolton:1995vx}. However, the latter notion is defined in the context
of compact Lie groups.

We have
\begin{proposition}\label{pro:cyc-simp}
	 Assertion \eqref{def:cyc5} is equivalent to: for all
$\beta$ and $\alpha$ in $Z$, we have
\begin{equation}
  f^*(\omega_\alpha\ww\omega_\beta)=0. \label{def:cyc5a}
\end{equation}
\end{proposition}
\begin{proof} We begin with an observation on cyclic roots. Let  $\alpha_i\in Z$, and assume that   $\alpha_0+\alpha_1=\alpha_2+\alpha_3$  is a root. Then $\{\alpha_0,\alpha_1\}=\{\alpha_2,\alpha_3\}$. Indeed, if all roots $\alpha_i$ are simple, this follows from the fact that simple roots form a free system. If one of the $\alpha_i$, say $\alpha_0$, is $-\eta$ then $\alpha_1$ is simple and  $\alpha_0+\alpha_1$ is negative. From it follows that either $\alpha_2$ or $\alpha_3$ is $-\eta$. Thus $\{\alpha_0,\alpha_1\}=\{\alpha_2,\alpha_3\}$.

 From this we deduce that $f^*(\phi\ww\phi)=0$ implies that for all $\alpha$ and $\beta$ in $Z$, then $f^*(\omega_\alpha\ww\omega_\beta)=0$. Indeed $f^*(\omega_\alpha\ww\omega_\beta)$ is zero if $\gamma=\alpha+\beta$ is not a root and otherwise takes values in $\mc G_{\gamma}$.

\end{proof}

\subsubsection{The reality condition}

Let $\ms G_0$ be the split real form associated to the Cartan subalgebra and its system of positive roots .

\begin{proposition} 
  Assume that $f:\Sigma\to X$ is a cyclic surface. Then the image of
  $\Sigma$ lies in a $\ms G_0$-orbit in $X$.
\end{proposition}
\begin{proof}
  By Assertion \eqref{def:cyc4} 
  $f^*(\lambda(\omega))=f^*(\omega)$. In other words $\CG{\T f(u)}=\T
  f(u)$ for all $u$ in $S$. Thus by Proposition \ref{real0}, 
  $f(\Sigma)$ is tangent to the foliation $\bf F$ defined by the
  "right" $\ms G_0$-orbits.
\end{proof}

\subsubsection{First example: the Fuchsian case}

Let $x=(\mk h,\dep,\rho,\lambda)$ be a point in $\X$. 
Let $\ms S$ be the principal  $\ms{SL}_2$ in $\ms G_0$ associated to $x$. By definition, the {\em Fuchsian surface} though $x$ is the orbit of $\ms S$.  This gives our first examples of cyclic surfaces.
\begin{proposition}\label{fuchs}
If $S$ is a Fuchsian surface in $\X$, then $S$ is a cyclic surface such that  $\left.\omega_\eta\right\vert_S=0$. 
\end{proposition}
\begin{proof} By construction, the Lie algebra of the complexification of $\ms S$ is generated by
$(a,X,Y)$, where
\begin{align*}
  a&=\frac{1}{2}\sum_{\alpha\in\Delta^+}{\rm
    h}_\alpha=\sum_{\alpha\in\Pi}r_\alpha\cwdot {\rm h}_\alpha,\cr
  X&= \sum_{\alpha\in\Pi}\sqrt{r_\alpha} {\rm x}_\alpha,\ Y=
  \sum_{\alpha\in\Pi}\sqrt{r_\alpha} {\rm x}_{-\alpha}.
\end{align*}
Thus, by Proposition \ref{prop:lS}, the Lie algebra of $\ms S$ is generated by $X-Y, iX+iY, ia$. Since $ia$ belongs to $\mk h$ and generates a compact subgroup, $ia\in\mk t$. 

Therefore, the orbit of $\ms S$ in $\X$ is with values in the complex 2-dimensional distribution $\mc W$ such that $\mc V\defeq\omega(\mc W)$ is generated by $X-Y, iX+iY$. Moreover, since $\ms S$ is real, the orbit of $\ms S$ in $\X$ is tangent to the  real 2-dimensional  distribution $\mc W_0$ such that 
$$
\omega(\mc W_0)\defeq\{u\in\mc V\mid \lambda(u)=u\}\eqdef\mc V_0.
$$
In particular, we now observe that 
$$
\mc V_0\subset \mc V\subset \mc Q\defeq \bigoplus_{\alpha\in\Pi}\mc G_\alpha\oplus\bigoplus_{\alpha\in\Pi}\mc G_{-\alpha}  \subset \mc G_Z\oplus \mc G_{Z^\dagger}.$$
Observe that $\mc V$ is fixed pointwise by $-\rho$, thus if $\phi$ and $\phi^\dagger$ is the projection from $\mc V$ to $\mc G_Z$ and $\mc G_{Z^\dagger}$ respectively, $-\rho(\phi)=\phi^\dagger$.
Thus a Fuchsian surface is a cyclic surface on which $\omega_\eta$ vanishes.
\end{proof}

\subsection{From cyclic surfaces to Higgs bundles}

We emphasise that the result of the next paragraph is local: the
surface $\Sigma$ is not assumed to be closed. For a cyclic map $f$
from $\Sigma$ to $\X $, we write $\Phi=f^*\phi$,
$\Phi^\dagger=f^*\phi^\dagger$. We denote by $\widehat\rho$ the
pullback of $\rho$ on $f^*\mc G$. By the definition of cyclic maps
$\Phi^\dagger=-\widehat\rho(\Phi)=\Phi^*$.
We also denote by $\widehat\nabla$ the induced connection $f^*\nabla$
on $f^*\mc G$

\begin{proposition}\label{prop:Higgs-cyc}
  Let $f$ be a cyclic map from $\Sigma$ to $\X $.  Then
  \begin{enumerate}
  \item there exists a unique complex structure on the surface so
    that $\Phi$ is of type $(1,0)$ and $\Phi^\dagger$ is of type
    $(0,1)$. \label{higgs-cyc1}
  \item the data $\Psi=(f^*\mathcal G, \Phi)$ defines a Higgs bundle
    whose Hopf differential is zero. \label{higgs-cyc2}
  \item The pair $(\widehat\nabla,\widehat\rho)$ on $f^*\mathcal G$ is the
    solution of the self duality equations: \label{higgs-cyc3}
    \begin{align}
      \widehat\nabla\widehat\rho&=0,\\
      R^{\widehat\nabla}&=2\cwdot\Phi\ww\Phi^*\label{higgsR},\\
      \d^{\widehat\nabla}\Phi &=0,\label{higgs2}\\
      \d^{\widehat\nabla}\Phi^* &=0\label{higgs2*}.
    \end{align}
  \item Finally, if $H$ is the Hitchin map from the space of Higgs
    bundles to the space of holomorphic differentials, then $H(\Psi)$ is a holomorphic
    differential of highest possible degree. \label{higgs-cyc5}
  \end{enumerate}
\end{proposition}
As a corollary, we get
\begin{corollary}\label{coro:Higgs-cyc} Let $f$ be a cyclic map.  Let $p$ be the projection
  from $\X $ to the symmetric space $\ms S(\ms G)$, then $p\circ f$ is
  a minimal surface. Moreover
  \begin{align}
  \operatorname{Area}(p\circ f)&=i\cwdot\int_S\Kill{\Phi}{\Phi^\dagger}\label{area-phi}
  \end{align}

\end{corollary}

\noindent{\em Proof of Proposition \ref{prop:Higgs-cyc}. }
  Since  $f^*\omega_\alpha$ never vanishes for any simple root $\alpha$, it follows
  that there exists exactly one complex $J_\alpha$ structure
  so that $f^*\omega_\alpha$ is of type $(1,0)$. We also know by Proposition \ref{pro:cyc-simp} that
  $f^*\omega_\alpha\ww f^*\omega_\beta=0$ for every pair $(\alpha,\beta)$ of simple roots. If $\alpha$ and $\beta$ are simple and $\alpha+\beta$ is a root, it follows from the fact that both $f^*\alpha$ and $f^*\beta$ are isomorphisms  that 
  $J_\alpha=J_\beta$. Then using the connectedness of the Dynkin diagram we get that for every pair $(\alpha,\beta)$,   $J_\alpha=J_\beta$.  This proves the uniqueness and shows that there
  exists a  complex structure such that for every simple root
  $\alpha$, $f^*\omega_\alpha$ is of type $(1,0)$. It remains to
  understand the type of $f^*\omega_{-\eta}$.
  
  Since $f^*\phi\ww f^*\phi=0$, decomposing along roots we obtain that
  for all simple root $\alpha$,
  \begin{equation}\label{equ:type1}
    f^*\omega_\alpha\ww f^*\omega_{-\eta}=0
  \end{equation}
  Since there exist a simple root $\alpha$ so that $\eta-\alpha$ is a
  root and in particular
$$
[\mathcal G_{\alpha},\mathcal G_{-\eta}]\not=0.
$$
Equation \eqref{equ:type1} implies that $f^*(\omega_{-\eta})$ is
of type $(1,0)$.  We thus obtain that $\Phi$ is of type $(1,0)$ and by
the reality condition that
$$
\Phi^\dagger=\Phi^*,
$$
is of type $(0,1)$.  This finishes the proof of statement
\eqref{higgs-cyc1}.  Statement \eqref{higgs-cyc2} is just an immediate consequence of the previous statement.

Let us now prove statement \eqref{higgs-cyc3}. Let $f$ be a cyclic
map. Recall that for a cyclic surface
\begin{align}
  f^*(\omega\ww\omega)&=\left(\Phi+\Phi^*\right)\ww \left(\Phi+\Phi^*\right)=2\, \Phi\ww\Phi^*\in\Omega^2(\Sigma,
\widehat{\mathcal H}).
\end{align}
Thus the curvature equation \eqref{curvfund}
$$
R^{\widehat{\nabla}}+\d^{\widehat{\nabla}}\Phi+\d^{\widehat{\nabla}}\Phi^\dagger+2\cwdot\Phi\ww\Phi^\dagger=0,
$$
yields the self duality field equation by taking the projections,
namely $\pi_0$ for the first equation and $\pi$ and $\pi^\dagger$ for
the two last equations:
\begin{align*}
  R^{\widehat{\nabla}}+2\cwdot\Phi\ww\Phi^*&=0,\\
  \d^{\widehat{\nabla}}\Phi&=0,\  \
  \d^{\widehat{\nabla}}\Phi^*=0.
\end{align*}

Finally, statement \eqref{higgs-cyc5} follows from Hitchin's
construction in Section 5 of \cite{Hitchin:1992es} (see also Baraglia
\cite{Baraglia:2010vi}) and Proposition \ref{InvPoly}.

\qed
\vskip 0.5 truecm
\noindent{\em Proof of Corollary \ref{coro:Higgs-cyc}. }
 For a
  smooth map $g: M\to N$ between manifolds, we consider $\T g$ as an
  element of $\Omega^1\left(M,g^*(\T N)\right)$. From Proposition \ref{MaurerSym}, $\T_\mathbb C \ms S(\ms
  G)$ is identified with the canonical $\mk g$ bundle over $\ms S(\ms
  G)$. As a consequence of the second assertion of Proposition
  \ref{MaurerSym}, we have
  \begin{align*}
    \frac{1}{2}(\Phi+\Phi^*)&=\T_\mathbb C (p\circ f),
  \end{align*}
and thus
  \begin{align*}
    \T^{(1,0)}_\mathbb C (p\circ f)=\frac{1}{2}\Phi,&\ \ \
    \T^{(0,1)}_\mathbb C (p\circ f)=\frac{1}{2}\Phi^*.
  \end{align*}
  Now the equation $\overline{\partial}\Phi=0$ says that
  $p\circ f$ is harmonic (See \cite{Donaldson:1987} and Proposition
  8.1.2. of \cite{Labourie:2005a}). Moreover by the last assertion of
  the Proposition \ref{prop:Higgs-cyc}, the Hopf differential of $f$
  is zero and thus $p\circ f$ is a minimal mapping (See Proposition
  8.1.4 of \cite{Labourie:2005a}). Equation \eqref{area-phi} follows at once from Proposition \ref{ener} and the fact that the energy of a minimal mapping is the area.
\qed


\subsection{From cyclic Higgs bundles to cyclic surfaces}

In this section, contrarily to the previous section, where the
construction was purely local, the surface is now assumed to be
closed. The main result of this section is

\begin{theorem}\label{Higgs2cyc} Let $(\widehat{\mc G},\Phi_{\rm q})$ be a cyclic Higgs
  bundle over a closed surface $\Sigma$. Then there exists a unique
  cyclic map $f$ from $\Sigma$ to $\X $ such that
  \begin{align}
    \widehat{\mc G}&=f^*(\mc G)\cr \Phi_{\rm q}&= f^*(\phi).
  \end{align}
\end{theorem}

\begin{proof}
  Let $(\widehat{\mc G},\Phi_{\rm q})$ be a cyclic Higgs bundle. Let $(\nabla,\widehat\rho)$ be the solution
  of the self duality equations. By  Proposition
  \ref{hpara}, the corresponding vector bundle $\widehat{\mc H}$ is parallel. Thus, following
  Hitchin Theorem \ref{th:H2}, the associated involution
  $\widehat{\sigma}$ is parallel. Thus the Hitchin--Kostant quadruple
  $(\widehat{\mc H},\widehat\Delta^+,\widehat{\rho},\widehat{\lambda})$ is parallel.
 Observe now that by Proposition \ref{HiggsCyclic},
$$
\Omega\defeq \Phi_{\rm q}+\Phi^*_{\rm q}\in
\Omega^1\left(\Sigma,f^*\left({\mc G}_Z\right)\oplus f^*\left({\mc
    G}_{Z^\dagger}\right)\right).
$$
In particular
$$
\forall u\in\widehat{\mc H},\ \
\widehat{\rho}(u)=\widehat{\sigma}(u)=u\implies\braket{u|\Omega}=0.$$
Recall finally that from the self duality equations $\nabla+\Omega$ is
flat.

Thus we can apply Proposition \ref{VecBundMan}, to obtain a map $f$
from $\Sigma$ to $\X $ so that
\begin{align}
  \widehat{\mc G}&=f^*(\mc G)\cr \Omega &= f^*(\omega).
\end{align}
It remains to prove that $f$ is cyclic. This follows at once from the
following three facts
\begin{enumerate}
\item By construction, we have that $\Phi_{\rm q}=f^*(\phi)$,
  $\Phi^*_{\rm q}=f^*(\phi^\dagger)$, $f^*\omega_0=0$,
  $f^*\omega_1=0$, and $f^*\omega_\alpha$ never vanishes for all simple roots.
\item By Equation \eqref{eq:PhiSigma}, $\sigma(\Phi_{\rm q})=-\Phi_{\rm q}$ and
  since $\rho$ commutes with $\sigma$, $\sigma(\Phi^*_{\rm q})=-\Phi^*_{\rm q}$.
\item Moreover $\Phi_{\rm q}\ww\Phi_{\rm q}$ is of type $(2,0)$, hence
  vanishes.
\end{enumerate}
We thus have verified Assertions \eqref{def:cyc1}, \eqref{def:cyc2},
\eqref{def:cyc3}, \eqref{def:cyc6}, \eqref{def:cyc4} and \eqref{def:cyc5} of the
definition of cyclic surfaces. \end{proof}
\subsection{Cyclic surfaces as holomorphic curves}
The purpose of this section to give another interpretation of cyclic
surfaces as holomorphic cuves in some quotient of $\ms
G_0$.

We use the notation of the previous section. Let us revisit the
definition of cyclic surfaces.  Let us first consider the complex
distributions in $\X $ given by $\mc V$ with $\omega(\mc V)=\mc G_Z$ with the complex
structure $J_0$ given by the multiplication by $i$, as well as $\mc V^\dagger$ with $\omega(\mc V^\dagger)=\mc
G_{Z^\dagger}$ with the complex
structure $J_0$ given by the multiplication by $-i$. Let $\mc W$ be the
complex distribution given by
$$
\mc W=\mc V\oplus \mc V^\dagger.
$$
Then the complex conjugation $\rho$ becomes a complex involution of
$\mc W$, and the Hitchin involution $\sigma$, preserving both $\mc V$ and $\mc
V^\dagger$, is also an antilinear involution.  We now consider the
subdistribution
$$
\mc S=\{u\in\mc W\mid \sigma(u)=-u,\ \ \rho(u)=-u\}.
$$
Observe that $\mc S$ is a subdistribution of $\T {\bf F}$ (see
Proposition \ref{real0}). Then we have

\begin{proposition}
  A cyclic surface is a surface everywhere tangent to $\mc S$ and
  whose tangent space is complex. Conversely, a surface everywhere
  tangent to $\mc S$ and whose tangent space is complex satisfies all the conditions for being a cyclic
  surface, except for the open condition \eqref{def:cyc6} of Definition \ref{cyc-map}.
\end{proposition}
\begin{proof} The proof is just linear algebra. If $\Sigma\hookrightarrow   \X$ is a cyclic surface, then for all $u$ tangent to $\Sigma$
\begin{align}
  \omega(u)&=\phi(u)+\phi^\dagger(u),\cr
   \lambda(\omega(u))&=\omega(u),\cr
   \rho(\phi(u))&=-\phi^\dagger(u).
  \end{align}
It follows that a cyclic surface is tangent to $\mc S$. Let now $J$ be the complex structure on $\Sigma$ so that $\Phi$ is of type $(1,0)$. We obtain 
that \begin{align}
  \omega(J\cdot u)&=\phi(J\cdot u)-\rho(\phi(J\cdot u))\cr
  &=i\,\phi(u)-\rho(i\,\phi(u))\cr
    &=i\,\phi(u)+i\,\rho(\phi(u))\cr
       &=i\,\phi(u)-i\,\phi^\dagger(u)\cr
       &=J_0\cdot\omega(u)\label{holom1}.
       \end{align}
In other words, $\T\Sigma$ is a complex subspace of $\mc W$.
    
    Conversely, assume  $\T\Sigma$ is a complex subspace of $\mc W$. Let us equip $\Sigma$ with the induced complex structure. Then by construction, for all $u\in\T\Sigma$
    \begin{align}
\omega_1(u)&=\omega_0(u)=0,\cr
\lambda(u)&=u.
\end{align}
Since $\omega=\phi+\phi^\dagger$ is fixed by $-\rho$ which exchanges $\mc G_Z$ and $\mc G_{Z^\dagger}$, it follows that $\rho(\phi)=-\phi^\dagger$. Finally, since $\phi(J_0u)=i\cwdot\phi(u)$, it follows that $\phi\ww\phi=0$ on $\Sigma$. In particular, $\Sigma$ is a cyclic surface.
\end{proof}
 
\section{Infinitesimal rigidity}\label{sec:IR}

In this section, we prove the infinitesimal rigidity for closed cyclic surfaces. The important corollary for us is Theorem
 \ref{Cyclicinjects} that we restate here.
\begin{theorem}\label{IFI0co}
  The map $\Psi:\mc E_{m_\ell+1}\to \mc H(\Sigma,\ms G_0)$ is an immersion.
\end{theorem}
We exploit the fact that cyclic surfaces are solutions of a Pfaffian system, which means that a certain family of forms vanishes on them, as well as a reality condition.

After a preliminary section on Pfaffian systems, we define infinitesimal variation and state our main result, Proposition \ref{IFI0}, in this language.

We prove Theorem \ref{IFI0co} as a corollary of Proposition \ref{IFI0}  in Paragraph \ref{trans}.

The proof of Proposition \ref{IFI0} occupies           most of the sequel and proceeds through obtaining formulas for the derivatives of the infinitesimal variation and a Böchner type formula.

\subsection{Preliminary: variation of Pfaffian systems}

In this section, totally independent on the rest, we explain a useful
proposition that we shall use in the sequel of the proof.

We shall consider the following setting. Let $\mc E_i$ be  vector
bundles over a manifold $M$ equipped with a connection $\nabla$. Let
$\Omega\defeq\left(\Omega_1,\ldots,\Omega_n\right)$ be a family of
forms $\Omega_i$ with values in $\mc E_i$.

\begin{definition}
  A submanifold $N$ of $M$ is {\em a solution of the Pfaffian system}
  defined by $\Omega$ if for  all $i$, $\Omega_i$ vanishes on $N$.
\end{definition}
From now on, by taking $\mc E=\bigoplus \mc E_i$, we may as well assume that all $\mc E_i$ are the same and equal to some vector bundle $\mc E$.
If $\mc E$ is a trivial line bundle, so that $\Omega_i$ are ordinary
forms, then we say the Pfaffian system is {\em elementary}. We can
always reduce any system to an elementary one, by choosing a local
trivialisation of $\mc E$ given by local sections $(v_\alpha)$, then
the {\em associated elementary Pfaffian system} in the trivialisation
is $(\Omega_i^\alpha)$ where
\begin{align}
  \Omega_i&=\sum_\alpha\Omega_i^\alpha\cwdot v_\alpha.
\end{align}

\subsubsection{Deformation of Pfaffian systems}
Let $F=(f_t)$ be a 1-parameter smooth family of deformations of maps
from $N$ to $M$ so that $f_0$ is the identity.  Let
\begin{align}
  \xi &=\left.\frac{\d}{\d t}\right\vert_{t=0}f_t.
\end{align}
Thus $\xi$ is a vector field along $N$, called the {\em tangent vector
  to the family $(f_t)$}. .
\begin{definition}
  The family $(f_t)$ is a {\em first order deformation} of the
  Pfaffian solution $N$ if, for all $i$
  \begin{align}
    \left.\frac{\d}{\d
        t}\right\vert_{t=0}f_t^*\Omega_i&=0,\label{def:Pfaff}
  \end{align}
  where using $\nabla$, we have identified $f_t^*(\mc E)$ with
  $f_0^*(\mc E)$ for all $t$.
\end{definition}
We observe that the definition does not depend on the choice of
$\nabla$: indeed, equivalently, $(f_t)$ is a first order deformation,
if and only  if it is a first order deformation for all elementary
associated Pfaffian system in local trivialisation.

Let us introduce the following definition

\begin{definition}
  A vector field $\xi$ along a solution of a Pfaffian system
  $\Omega=(\Omega_1,\ldots,\Omega_n)$ is {\em an infinitesimal
    variation fo the Pfaffian system} if for all $i$
  \begin{align}
    \left.i_\xi\d^\nabla\Omega_i\right\vert_N&=-
    \left.\d^\nabla\left(i_\xi\Omega_i\right)\right\vert_N.\label{Pfaff00}
  \end{align}

\end{definition}

The following relates the two definitions and will be an important
technical tool
\begin{proposition}
  Assume that $\xi$ is a tangent vector to a family of first order
  deformation of the Pfaffian system. Then $\xi$ is an infinitesimal
  variation of the Pfaffian system: for all $i$,
  \begin{align}
    \left.i_\xi\d^\nabla\Omega_i\right\vert_N&=-
    \left.\d^\nabla\left(i_\xi\Omega_i\right)\right\vert_N.\label{Pfaff001}
  \end{align}
\end{proposition}

\begin{proof} It is enough to assume that $n=1$, that is $\Omega=(\Omega)$.
  Assume first that $\nabla$ is the trivial connection. We consider
  $(f_t)$ as a map $F$ from $P\defeq N\times[0,1]$ to $N$. Let $\partial_t$
  be the canonical vector on $P$ associated to the flow
  $\phi_t:(n,s)\mapsto (n,s+t)$. Let also $J$ be the injection
  $n\mapsto (n,0)$ from $N$ into $P$. Let finally  $\Theta=F^*\Omega$.  Observe first that for any form $\alpha$,
  \begin{align}
    J^*(i_{\partial_t}F^*\alpha)=J^*(F^*i_{F_*\partial_t}\alpha)=J^*(F^*(i_{\xi}\alpha))=f_0^*(i_\xi\alpha).\label{infipfaf1}
  \end{align}
Since $\xi$ is is a tangent vector to a family of first order
  deformation of the Pfaffian system, we have 
  \begin{align}
    J^*\Theta=0,\ \ \ & \ \ \ J^*L_{\partial_t}\Theta =0.
  \end{align}
  By the Lie--Cartan formula,
  \begin{align}
    L_{\partial_t}\Theta &=\d
    i_{\partial_t}\Theta+i_{\partial_t}\d\Theta.
  \end{align}
Using Equation \eqref{infipfaf1}, we get
 \begin{align}
   0 &=J^*\d
    i_{\partial_t}\Theta+J^*i_{\partial_t}\d\Theta\cr
    &=
    J^*\d
    i_{\partial_t}F^*\Omega+J^*i_{\partial_t}\d F^*\Omega\cr
    &=f_0^*\left(\d 
    i_{\partial_\xi}\Omega+i_{\partial_\xi}\d\Omega\right)
  \end{align} Since $f_0$ is the identity, the last equation yields
    \begin{align}
    0&=\d
    i_{\xi}\Omega+i_{\xi}\d\Omega.
  \end{align}
  Thus the conclusion of the proposition holds when $\nabla$ is trivial.
  Assume now that $\nabla$ is not trivial. Let
  $x_0\in N$. We can find locally a base $(v_\alpha)$ of $\mc E$
  giving a local trivialisation, such that $\nabla v_\alpha=0$ at
  $x_0$. Let us write
  \begin{align}
    \Theta&=\sum_{\alpha}\Theta_\alpha\cwdot v_\alpha.
  \end{align}
  Observe that $N$ is also a solution of the Pfaffian system defined
  by $(\Theta_\alpha)$ and that $\xi$ is an infinitesimal deformation
  of that Pfaffian system. Thus, at $x_0$,
  \begin{align}
    i_\xi\d^\nabla\Theta&=\sum_\alpha
    i_\xi\d\Theta_\alpha\cwdot v_\alpha\cr &=-\sum_\alpha
    \d(i_\xi\Theta_\alpha)\cwdot v_\alpha\cr
    &=-\d^\nabla(i_\xi\Theta),
  \end{align}
  where in the first equality we used that $\nabla v_\alpha=0$ at
  $x_0$, in the second we used Equation \eqref{Pfaff00} for the
  Pfaffian system $(\Omega_\alpha)$ and finally in the last equality
  we used $\nabla v_\alpha=0$ at $x_0$ again.
\end{proof}

The careful reader could check that Equation \eqref{Pfaff00} is
independent of the choice of $\nabla$ if $\Omega$ vanishes along $N$.
\subsection{Cyclic surfaces as solutions of a Pfaffian system}
\begin{definition}
  The {\em cyclic Pfaffian system} is the family of forms
  $$\Lambda\defeq(\omega_0,\omega_1,\phi\ww\phi,\phi+\rho(\phi^\dagger)).$$
\end{definition}

By definition a cyclic surface is a solution of the Pfaffian system $\Lambda$. Observe that for a solution of the Pfaffian system $\Lambda$, $\phi^\dagger\ww\phi^\dagger$ also vanishes. 
 Another form vanishes for cyclic surfaces:  let $\mc H_0$ be the orthogonal in
$\mc H$ with respect to $\rho$ of the subdistribution $\mc T$
corresponding to the Lie algebra of $\ms T$. We have the orthogonal
decomposition
\begin{align*}
  \mc H&=\mc T\oplus\mc H_0.
\end{align*}
We can thus write
\begin{align*}
  \pi_0&=\pi_t+\check\pi_0.
\end{align*}
where $\check\pi_0$ and $\pi_{t}$ are the orthogonal projections
respectively on $\mc H_0$  and $\mc T$. Since
$\check\pi_0\left(R^\nabla\right)=0$, it follows from the self
duality equations \eqref{higgsR} that
$\check\pi_0(\phi\ww\phi^\dagger)\VS=0$.

\subsection{Infinitesimal deformation of cyclic surfaces}
Let $\xi$ be a vector field along a cyclic surface $f$.

\begin{definition}
  We say $\xi$ is an {\em infinitesimal deformation of cyclic
    surfaces}, if $\xi$ is an infinitesimal deformation of the cyclic
  Pfaffian system and if $\xi$ is {\em real}, that is $\overline\xi=\xi$.  We say $\xi$ is an {\em infinitesimal deformation of closed cyclic
    surfaces}, if furthermore the underlying surface is closed.\end{definition}

In the rest of this section, $\xi$ will be a fixed infinitesimal
variation of cyclic surfaces. We will also consider, for the sake of simplicity, the surface $\Sigma$ as a submanifold of $\X$, statement which is locally true.
The following is our main result.

\begin{proposition}\label{IFI0}
  Let $\xi$ be an infinitesimal deformation of a closed cyclic
  surface. Assume that there exists a simple root such that
  $\omega_\alpha(\xi)=0$, then $\xi=0$.
\end{proposition}
In this proposition, $\omega_\alpha$ is defined in decomposition \eqref{MC-dec}.

\subsubsection{The cyclic decomposition of an infinitesimal deformation}
The cyclic decomposition of $\xi$ is given by
\begin{align}
  i_\xi\omega=\zeta_0 +\zeta_1 +\zeta+\zeta^\dagger\label{xidecomp}
\end{align}
where
\begin{align*}
  \zeta_0\defeq i_\xi \omega_0\in \mc H\ \ & \ \
  \zeta_1\defeq i_\xi\omega_1\in \mc G_1,\\
  \zeta\defeq i_\xi\phi\in\mc G_Z ,\ \ & \ \ \zeta^\dagger\defeq
  i_\xi\phi^\dagger\in\mc G_{Z^\dagger}.
\end{align*}
Equation \eqref{eq:MCcyc0} implies that actually
\begin{align}
  \zeta_0\in\mc H_0. \label{realcondzeta}
\end{align}
\subsubsection{Reality condition}
We assume that $\xi$ is a real vector, meaning that $\CG{\xi}=\xi$, that is by definition $\lambda(i_\xi\omega)=i_\xi\omega$. It
then follows 
\begin{proposition}
We have
\begin{align}
  \lambda(\zeta)=\zeta^\dagger,\ \ \lambda(\zeta_0)=\zeta_0,\ \
  \lambda(\zeta_1)=\zeta_1.\label{realcond2}
\end{align}
Moreover
\begin{align}
  \rho(\zeta_0)=-\zeta_0.\label{rhozero}
\end{align}
\end{proposition}
\begin{proof}
The first equality in \eqref{realcond2} comes from the fact that $\sigma$ preserves $\mc G_Z$ and  $\mc G_{Z^\dagger}$ respectively (last statement of Proposition \ref{sig-involution})
and $\rho$ exchanges them. The second and third follows from the fact that $\pi_0$ and $\pi_1$ commutes with $\lambda$.

For the  equality \eqref{rhozero}, remark that $\rho$ is an involution that globally preserves $\mc T$, hence $\mc H_0$, as well as 
$$
\mc H_\lambda\defeq\{u\in\mc H_0\mid \lambda(u)=u\}.
$$
Any fixed vector by  $\rho$  in $\mc H_\lambda$, belongs to $\mc T$ hence is null. It follows that $\rho$ acts as $-1$ on $\mc H_\lambda$.
\end{proof}

\subsubsection{The root space decomposition}
We also write the following decomposition of $\xi$ as
\begin{align*}
  i_\xi\omega&=\zeta_0+\sum_{\alpha\in \Delta}\zeta_\alpha, \ \ \hbox{ where
  }\ \ \zeta_\alpha\in\mc G_\alpha.
\end{align*}

\subsubsection{Proof of the transversality of the Hitchin map}\label{trans}

In this paragraph, we prove Theorem \ref{IFI0co}, assuming Proposition \ref{IFI0}. The proof is standard. 
As a standard notation if $(x_t)_{t\in ]-1,1[}$ is a $C^1$-curve in  a manifold $M$, we write
\begin{align}
  \dt x_0\defeq\left.\frac{\d}{\d t}\right\vert_{t=0}x_t \in \T_{x_0}M.
\end{align}

Let $(J_t,{\rm q}_t)_{t\in]-1,1[}$ be a family of elements of $\mc E_{m_\ell +1}$. By Theorem \ref{Higgs2cyc}, we associate to $(J_t,{\rm q}_t)$ a homomorphism $\delta_t$ of $\pi_1(\Sigma)$ in $\ms G_0$ and a $\delta_t$-equivariant cyclic map $f_t$ from $\Sigma$ from $\X$. Then by definition $\Psi(J_t,{\rm q}_t)=[\delta_t]$ is the equivalence class (by conjugation) of $\delta_t$.

Observe first that $\xi(s)\defeq \dt f_0(s)$ is an infinitesimal deformation of cyclic surfaces in $\X$.

We want to prove the local injectivity of $\Psi$. Let us thus assume that $$\dt{[\delta_0]}=0.$$ Since the smooth manifold  $\mc H(\Sigma,\ms G_0)$ only consists of irreducible representations, $$\T_{\delta_0}\mc H(\Sigma,\ms G_0)=H^1_{\delta_0}(\Sigma,\mk g).$$
 Thus after possibly conjugating the family $(\delta_t)$ by a family $(g_t)$ of elements of $\ms G_0$ and replacing $(f_t)$ by $(g_t\cdot f_t)$ we obtain that 
\begin{align}
 \forall s\in\tilde\sigma,\,\forall \gamma\in\pi_1(\Sigma),\ \ \dt f_0(\gamma(s))=\delta_0(\gamma)\cwdot \dt f_0(s).
\end{align}
In particular, $\xi(s)\defeq \dt f_0(s)$ is an infinitesimal deformation of closed cyclic surfaces in $\delta_0(\pi_1(\Sigma))\backslash\X$. 

Let us fix a simple root $\alpha$. Recall that by definition of cyclic surfaces,  $f_0^*\omega_\alpha$ is a bijection from $\T\Sigma$ to $f_0^*(\mc G_\alpha)$. Thus, let $\nu$ be the vector field along $\Sigma$ so that $\zeta_\alpha=f_0^*\omega_\alpha(\nu)$. Since every vector field tangent to the surface is an infinitesimal deformation of cyclic surfaces, $\T f_0\left(\nu\right)$ is an infinitesimal deformation of cyclic surfaces.

Let $\widehat\xi\defeq\xi-\T f_0\left(\nu\right)$,  by construction $\widehat\xi$ is an infinitesimal deformation of cyclic surfaces whose component along $\mc G_\alpha$ is zero. Applying Proposition \ref{IFI0}, we obtain that $\widehat\xi=0$. It remains to prove  that $\dt{J_0}=0$ and $\dt {\rm q}_0=0$: that will conclude the proof of the injectivity of  $\T\Psi$.  Choosing locally a Chevalley basis ${\rm x}_\alpha$ of $\mc G_\alpha$, we may write $\omega_\alpha=\Omega_\alpha\cdot {\rm x}_\alpha$ where $\Omega_\alpha\in\Omega^1(\Sigma,\mathbb C)$. Since $\widehat\xi=0$, it follows that for all roots $\alpha$, $\dt\Omega_\alpha=0$. Now for a simple root $\alpha$, $\Omega_\alpha$ is non zero and the complex structure on $\Sigma$ is characterised by the fact that $\Omega_\alpha$ is of type $(1,0)$. Thus $\dt{J_0}=0$. Similarly $\rm q$ is obtained by a formula involving the $\Omega_\alpha$, thus $\dt{\rm q}_0=0$.

\subsection{Computations of first derivatives}
From now on, we assume that $\xi$ is an infinitesimal deformation of cyclic surfaces.
We will first obtain expressions for the derivatives of $\zeta_0$ and $\zeta_1$ exploiting the fact that $\omega_0$ and $\omega_1$ vanish on cyclic surfaces. We will denote in the sequel
\begin{align*}
  \partial=\left(\d^\nabla\right)^{(1,0)}, \ \ \overline{\partial}
  &=\left(\d^\nabla\right)^{(0,1)}.
\end{align*}
\subsubsection{Vanishing of $\omega_0$ and the derivatives of
  $\zeta_0$}
Here, we exploit the fact that $\omega_0\VS=0$.
\begin{proposition}\label{ifil}
  We have \tfe
  \begin{align}
    {\partial}\zeta_0&=2\cwdot \check\pi_0(\zeta^\dagger\ww\phi)=2\cwdot \check\pi_0(\rho(\zeta)\ww\phi),
    \\
    \overline{\partial}\zeta_0&=2\cwdot
    \check\pi_0(\zeta\ww\phi^\dagger)=2\cwdot
    \check\pi_0(\rho(\zeta^\dagger)\ww\phi^\dagger).
  \end{align}
\end{proposition}
\begin{proof}  We have $\zeta_0=i_\xi\omega_0$. By the definition of infinitesimal variation, and using the fact that $\check\pi_0$ is parrallel, we have
  \tfe
  \begin{align*}
     \nabla\zeta_0=\d^\nabla i_\xi\omega_0=-i_\xi \d^\nabla\omega_0
    =-i_\xi
    \d^\nabla\check\pi_0\left(\omega\right)
    =
    -i_\xi\check\pi_0(\d^\nabla \omega).
  \end{align*}
   Thus, the curvature equation
  \eqref{curvfund} yields
  \begin{align*}
    \nabla\zeta_0     &=i_\xi\check\pi_0\left(R^\nabla+\omega\ww\omega\right).
  \end{align*}
Then by Proposition \ref{vec-bun-des}, we have that
  $\check\pi_0(R^\nabla)=0$. Thus  we get
  \begin{align}
    \nabla\zeta_0
    &=i_\xi\check\pi_0(\omega\ww\omega).\label{zeta091}
  \end{align}
Finally, since $\omega_1\VS=0$, it follows that
  $(i_\xi(\omega_1\ww\omega_1))\VS=0$.  Thus, combining Equations
  \eqref{zeta090} and \eqref{zeta091}, we get \tfe
  \begin{align*}
    \nabla\zeta_0&=2\cwdot
    i_\xi\check\pi_0(\phi\ww\phi^\dagger)=2\cwdot\check\pi_0(\zeta\ww\phi^\dagger+\zeta^\dagger\ww\phi).
  \end{align*}
  Using the fact that $\phi$ and $\phi^\dagger$
  are respectively of type $(1,0)$ and $(0,1)$, we get the first part of both equations in the proposition. To get the second part, we use that $\zeta_0=-\rho(\zeta_0)$ and $\phi=-\rho(\phi^\dagger)$. 
\end{proof}

\subsubsection{Vanishing of $\omega_1$ and the derivatives of
  $\zeta_1$}

We exploit the fact that $\omega_1\vert_\Sigma=0$.
\begin{proposition}\label{1ifi}
  We have \tfe
  \begin{align} {\partial} \zeta_1 &=2\cwdot\pi_1((\zeta_1+\zeta)\ww
    \phi),\cr \overline{\partial} \zeta_1
    &=2\cwdot\pi_1((\zeta_1+\zeta^\dagger)\ww \phi^\dagger).
  \end{align}\label{eq:1ifi}
\end{proposition}
\begin{proof}
  By the definition of infinitesimal variation and using that $\pi_1$ is parallel, we have
  \begin{align*}
     \nabla \zeta_1=\d^\nabla i_\xi\omega_1=-i_\xi\d^\nabla\pi_1(\omega)=-i_\xi\pi_1(\d^\nabla\omega).
  \end{align*}
  Thus, since $\pi_1(R^\nabla)=0$ the curvature equation \eqref{curvfund} yields
  \begin{align}
    \nabla \zeta_1
    &=i_\xi \pi_1(\omega\ww \omega).\label{ifi11}
  \end{align}
Observe also that we have \tfe
  \begin{align}
    i_\xi(\omega_1\ww\omega_0) =0,& \ \ \ \ i_\xi(\omega_1\ww\omega_1)
    =0. \label{ifi14}
  \end{align}
  Combining equations \eqref{pi1om}, \eqref{ifi11}, and \eqref{ifi14}
  we get \tfe
  \begin{align*}
    \nabla \zeta_1&=i_\xi\pi_1\left(2.\omega_1\ww\phi +
      2.\omega_1\ww\phi^\dagger+\phi\ww\phi
      +\phi^\dagger\ww\phi^\dagger\right)\cr &=2\cwdot\pi_1\left(\zeta_1\ww
      \phi+\zeta_1\ww\phi^\dagger+\zeta\ww\phi+\zeta^\dagger\ww\phi^\dagger)\right).
  \end{align*}
  Now we can decompose the last equation into types, using the fact that $\phi\VS$ is of
  type $(1,0)$ and $\phi^\dagger\VS$ is of type $(0,1)$ to get
  \begin{align*} {\partial} \zeta_1 &=2\cwdot\pi_1(\zeta_1\ww
    \phi+\zeta\ww\phi),\cr \overline{\partial} \zeta_1
    &=2\cwdot\pi_1(\zeta_1\ww
    \phi^\dagger+\zeta^\dagger\ww\phi^\dagger).
  \end{align*}
  The proposition now follows.
\end{proof}

\subsection{Again, computation of first derivatives}

So far we have obtained  direct information about the first derivatives of $\zeta_0$ and $\zeta_1$ using vanishing of the 1-forms $\omega_0$ and $\omega_1$. In this section, we obtain constraints on the derivatives of $\zeta$ and $\zeta^\dagger$ using the vanishing of 2-forms.

\subsubsection{A preliminary computation}

The next proposition does not use the fact that $\xi$ is an infinitesimal deformation of cyclic surfaces.
\begin{proposition}\label{ifi00}
  We have the following equality in $\Omega^*(\Sigma,\mc G)$
  \begin{align}
    (i_\xi\d^\nabla\phi)^{(0,1)}&=-2\cwdot\pi((\zeta_1+\zeta^\dagger)\ww\phi^\dagger),\label{ifi001}\\
    (i_\xi\d^\nabla\phi)^{(1,0)}&=-2\cwdot\zeta_0\ww\phi.\label{ifi002}
  \end{align}
  Symmetrically
  \begin{align}
    (i_\xi\d^\nabla\phi^\dagger)^{(0,1)}&=-2\cwdot\zeta_0 \ww\phi^\dagger,\label{ifi003}\\
    (i_\xi\d^\nabla\phi^\dagger)^{(1,0)}&=-2\cwdot
    \pi^\dagger((\zeta_1+\zeta)\ww\phi)\label{ifi004}.
  \end{align}

\end{proposition}
\begin{proof}
  First observe that using  Assertion \eqref{piom} of Proposition \ref{cyclic-dec}
  \begin{align}
    \d^\nabla \phi &=\d^\nabla\pi(\omega) =\pi(\d^\nabla\omega)\cr
    &=-\pi(\omega\ww\omega+R^\nabla)\cr
        &=-\pi(\omega\ww\omega)\cr
    &=-2\cwdot \omega_0\ww\phi-2\cwdot\pi(\omega_1\ww\phi^\dagger)\cr
       &-\pi(\phi^\dagger\ww\phi^\dagger)-\pi(\omega_1\ww\omega_1).\label{ifi-0}
  \end{align}
For a cyclic surface $\omega_i\VS=0$ for $i=0,1$, and thus
  for $i,j=0,1$
  \begin{align*}
    i_\xi(\omega_i\ww\omega_j)\VS=0.
  \end{align*}
Thus Equation \eqref{ifi-0} yields the following equality in
  $\Omega^*(\Sigma,\mc G)$
  \begin{align*}
    i_\xi\d^\nabla\phi&=-
    i_\xi\left(2\cwdot\omega_0\ww\phi
    +\pi\left(2\cwdot \omega_1\ww\phi^\dagger+\phi^\dagger\ww\phi^\dagger\right)\right)\cr
    &=-2\cwdot\zeta_0\ww\phi-2\cwdot\pi(\zeta_1\ww\phi^\dagger+\zeta^\dagger\ww\phi^\dagger).
  \end{align*}
  Since $\phi$ is of type $(1,0)$ and $\phi^\dagger$ is of type
  $(0,1)$ the previous equation yields the first part of the
  proposition, where in the second equality, we use that
  $\pi(\zeta_0\ww\phi)=\zeta_0\ww\phi$. The second part follows by symmetry.
\end{proof}
\subsubsection{Vanishing  of $\phi+\rho(\phi^\dagger)$ and the derivatives of $\zeta+\rho(\eta^\dagger)$ }
Let $\mu=\zeta+\rho(\zeta^\dagger)$, then
\begin{proposition}\label{dmu}
We have the following equality in $\Omega^*(\Sigma,\mc G)$.
\begin{align}
\partial\mu=4\cwdot\zeta_0\ww\phi.
\end{align}
\end{proposition}

\begin{proof}
Let $\beta=\phi+\rho(\phi^\dagger)$. Using Proposition \ref{ifi00}, we get
\begin{align}
(i_\xi\d^\nabla\beta)^{(1,0)}&=-2\cwdot\zeta_0\ww\phi-2\cwdot\rho(\zeta_0\ww\phi^\dagger)\cr
&=-2\cwdot\zeta_0\ww\phi+2\cwdot\rho(\zeta_0)\ww\phi\cr
&=-4\cwdot\zeta_0\ww\phi,
\end{align}
where we used Equation \eqref{rhozero} in the last equality.
Then, by the vanishing of $\beta$ along cyclic surfaces, we obtain
\begin{align}
(i_\xi\d^\nabla\beta)^{(1,0)}&=-(\d^\nabla i_\xi\beta)^{(1,0)}\cr
&=-\partial\mu.
\end{align}
This proves the result.
\end{proof}

\subsubsection{Vanishing of $\phi\ww\phi$ and the derivatives of $\zeta$.}

\begin{proposition}\label{phiphi}
  We have \tfe
  \begin{align}
    (\nabla \zeta)\ww\phi&=2\cwdot\phi\ww\pi((\zeta_1+\zeta^\dagger)\ww\phi^\dagger),\label{ifi211}\\
    (\nabla
    \zeta^\dagger)\ww\phi^\dagger&=2\cwdot\phi^\dagger\ww\pi^\dagger((\zeta_1+\zeta)\ww\phi)\label{ifi211s}.\end{align}
\end{proposition}

\begin{proof} Let $\Psi=\phi\ww\phi$.  By the definition of
  infinitesimal variation, \tfe\ holds

  \begin{align}
    \d^\nabla i_\xi\Psi&=-i_\xi\d^\nabla\Psi=-2\cwdot i_\xi
    (\d^\nabla\phi\ww\phi).
  \end{align}
  Recall that by Equation \eqref{higgs2} for a cyclic surface
  $\d^\nabla\phi\VS=0$. Thus the last equation yields (after a type
  decomposition)
  \begin{align}
    \d^\nabla i_\xi\Psi&=-2\cwdot \phi\ww
    (i_\xi\d^\nabla\phi)^{(0,1)}.\label{ifi22}
  \end{align}
  Then Equation \eqref{ifi001} from Proposition \ref{ifi00} yields
$$
\d^\nabla
i_\xi\Psi=4\cwdot\phi\ww\pi((\zeta_1+\zeta^\dagger)\ww\phi^\dagger).
$$
Now, $i_\xi\Psi=2\cwdot\zeta\ww\phi$. Thus the previous equation
combined with the fact that $\d^\nabla\phi\VS=0$ yields
$$
\nabla\zeta\ww\phi=2\cwdot\pi((\zeta_1+\zeta^\dagger)\ww\phi^\dagger)\ww\phi.
$$
A symmetric argument (using $\phi^\dagger\ww\phi^\dagger=0$ along $\Sigma$) yields the last statement.
\end{proof}

\subsection{Computation of second order derivatives}

We now combine the two previous sections to obtain formulas for the second derivatives of $\zeta_1$ and $\zeta_0$.

\begin{proposition}\label{ifi2B}
  We have
  \begin{align}
    \d^\nabla{\partial}\zeta_1
    &=4\cwdot\pi_1\left((\zeta_1\ww\phi^\dagger)\ww\phi\right),\\
    \d^\nabla\overline{\partial}\zeta_1
    &=4\cwdot\pi_1\left((\zeta_1\ww\phi)\ww\phi^\dagger\right).
  \end{align}
\end{proposition}
\begin{proof}
  By Proposition \ref{1ifi}
  \begin{align*}{\partial} \zeta_1 &=2\cwdot\pi_1((\zeta_1+\zeta)\ww
    \phi).\end{align*} Thus,
  \begin{align}\d^\nabla{\partial} \zeta_1
        &=2\cwdot\d^\nabla\pi_1(\zeta_1\ww\phi)+2\cwdot\d^\nabla\pi_1(\zeta\ww
    \phi).\label{ifi2B1}\end{align}
Let us first consider
  the derivatives of $\pi_1(\zeta_1\ww\phi)$. Since $\d^\nabla\phi=0$
  by Equation \eqref{higgs2}, we have, using Proposition \ref{1ifi}
  for the third equality,
  \begin{align}
    \d^\nabla(\pi_1(\zeta_1\ww\phi))&=\pi_1(\nabla \zeta_1\ww \phi)\cr
    &=\pi_1(\overline{\partial}\zeta_1\ww \phi)\cr
    &=2\cwdot\pi_1\left(\phi\ww\pi_1\left((\zeta_1+\zeta^\dagger)\ww\phi^\dagger\right)\right).\label{ifi2B2}
  \end{align}
  Similarly, now using Proposition \ref{phiphi}
  \begin{align}
    \d^\nabla(\pi_1(\zeta\ww\phi))&=\pi_1(\nabla \zeta\ww \phi)\cr
    &=2\cwdot\pi_1\left(\phi\ww\pi\left((\zeta_1+\zeta^\dagger)\ww\phi^\dagger\right)\right).\label{ifi2B2b}
  \end{align}
  Combining Equations \eqref{ifi2B2} and \eqref{ifi2B2b}, we get
  \begin{align}
    \d^\nabla{\partial}\zeta_1
    &=4.\pi_1\left(\phi\ww(\pi_1+\pi)\left((\zeta_1+\zeta^\dagger)\ww\phi^\dagger\right)\right).
  \end{align}
  Since by Proposition \ref{crucial}, \begin{align} [\mc G_1,\mc
    G_{Z^\dagger}]&\subset \mc G_1+\mc G_{Z},\cr [\mc G_{Z^\dagger},\mc
    G_{Z^\dagger}]&\subset \mc G_1+\mc G_{Z},
  \end{align}
  we get
  \begin{align}
    \d^\nabla{\partial}\zeta_1
    &=4.\pi_1\left(\phi\ww\left((\zeta_1+\zeta^\dagger)\ww\phi^\dagger\right)\right).\label{ifi2B10}
  \end{align}
  The Jacobi identity yields
  \begin{align*}
    (\zeta^\dagger\ww\phi^\dagger)\ww\phi&=\zeta^\dagger\ww
    (\phi^\dagger\ww\phi) +\phi^\dagger\ww(\zeta^\dagger\ww\phi)\in
    \Omega^2(\Sigma,\mc G_{Z^\dagger}).
  \end{align*}
  Thus
  \begin{align*}
    \pi_1\left((\zeta^\dagger\ww\phi^\dagger)\ww\phi\right)&=0.
  \end{align*}
  Thus in the end, Equation \eqref{ifi2B10} yields
  \begin{align}
    \d^\nabla{\partial}\zeta_1
    &=4.\pi_1\left(\phi\ww(\zeta_1\ww\phi^\dagger)\right).\label{ifi23}
  \end{align}
  The proof of the second equation of the proposition follows by
  inverting the role of $\phi$ and $\phi^\dagger$.
\end{proof}
\begin{proposition}\label{ifi2B0}
  We have
  \begin{align}
    \d^\nabla{\overline\partial}\zeta_0
    &=4\cwdot\check\pi_0((\zeta_0\ww\phi)\ww\phi^\dagger).
  \end{align}
\end{proposition}
\begin{proof}
By Proposition \ref{ifil}, we have 
  \begin{align}
\overline\partial\zeta_0
    &=\check\pi_0(\mu\ww\phi^\dagger).
\end{align}
By Proposition \ref{dmu}, and using the fact that $\d^\nabla\phi^\dagger=0$, it follows that
\begin{align}
\d^\nabla(\overline\partial\zeta_0)&=\check\pi_0(\partial\mu\ww\phi^\dagger)\\
    &=4\cwdot\check\pi_0((\zeta_0\ww\phi)\ww\phi^\dagger).
\end{align}
\end{proof}

\subsection{Proof of the infinitesimal rigidity}

Our goal is now to prove Proposition \ref{IFI0}. We will use freely in the sequel the fact that for $A=\check\pi_0,\pi_1 \hbox { or } \pi_0$
  $\int_\Sigma\Kill{u}{A(v)}=\int_\Sigma\Kill{A(u)}{v}$.

\subsubsection{First step}
\begin{proposition}\label{IFIp}
  Let $\xi$ be an infinitesimal deformation of a closed cyclic surface.  Then
  \begin{align}
    \zeta_1\ww\phi=\zeta_1\ww\phi^\dagger&=0\label{IFIp0},\\
    \nabla \zeta_1&=0\label{IFIp1},\\
    \pi_1(\zeta\ww\phi)\label{IFIp2}&=0.
  \end{align}\end{proposition}

\begin{proof}
  Let $\Sigma$ be a closed surface. By Proposition \ref{ifi2B},
  \begin{align*}
    \d^\nabla{\partial}\zeta_1
    &=4\cwdot\pi_1\left((\zeta_1\ww\phi^\dagger)\ww\phi\right).
  \end{align*}
   Denoting by $\Kill{\cwdot}{\cwdot}$ the Killing form, an integration yields
  \begin{align*} \int_\Sigma \Kill{
      \CK{\zeta_1}}{\d^\nabla{\partial}\zeta_1}
    &=4\cwdot\int_\Sigma\Kill{\CK{\zeta_1}}{\pi_1\left((\zeta_1\ww\phi^\dagger)\ww\phi
        \right)}.
        \end{align*} Observe now that since $\rho$
  preserves $\mk g_1$, $\pi_1\left(\CK{\zeta_1}\right)=\CK{\zeta_1}$. Thus for
  all $\kappa$,

$$
\int_\Sigma \Kill{\CK{\zeta_1}}{\pi_1 (\kappa)}=\int_\Sigma
\Kill{\CK{\zeta_1}}{\kappa}.
$$ 
Thus
\begin{align*}
  \int_\Sigma \Kill{ \CK{\zeta_1}}{\d^\nabla{\partial}\zeta_1}
  &=4\int_\Sigma\Kill{\CK{\zeta_1}}{(\zeta_1\ww\phi^\dagger)\ww\phi}.
\end{align*}
Using Equation \eqref{killexchange} and the fact that  $\phi^\dagger=-\CK{\phi}$, we get
\begin{align*}
  \int_\Sigma \Kill{ \CK{\zeta_1}}{\d^\nabla{\partial}\zeta_1}
  &=-4\int_\Sigma\Kill{\CK{\zeta_1}\ww\phi}{\zeta_1\ww\CK{\phi}}.
\end{align*}
An integration by part and Proposition \ref{Sign} finally yields
\begin{align*}
 0\leq  i\cwdot \int_\Sigma \Kill{ \CK{{\partial}\zeta_1}}{{\partial}\zeta_1}
  &=4i\cwdot\int_\Sigma\Kill{\CK{\zeta_1}\ww\phi}{\zeta_1\ww\CK{\phi}}\leq 0.
\end{align*}
It then follows that
\begin{align} {\partial}\zeta_1&=0\label{IFI21}\\
  \zeta_1\ww\phi^\dagger&=0\label{IFI22}
\end{align}
Symmetrically
\begin{align}
  \overline{\partial}\zeta_1&=0\label{IFI11}\\
  \zeta_1\ww\phi&=0\label{IFI12}.
\end{align}

Assertion \eqref{IFIp0} is just Equations \eqref{IFI12} and
\eqref{IFI22}.

Assertion \eqref{IFIp1} now follows from Equations \eqref{IFI11} and
\eqref{IFI21}.

Assertion \eqref{IFIp2} then follows from Proposition \ref{1ifi} and
Equation \eqref{IFI12}.\end{proof}
\subsubsection{Second step}

\begin{proposition}\label{IFIp22}
  Let $\Sigma$ be a closed surface. Let $\xi$ be an infinitesimal
  variation. Then
  \begin{align}
    \nabla\zeta_0&=0\\
    \zeta_0\ww\phi &=0\\
    \check\pi_0(\zeta\ww\phi^\dagger)&=0.
  \end{align}
\end{proposition}

\begin{proof}
Using Proposition \ref{ifi2B0}
  \begin{align}
    \int_S \Kill{\d^\nabla\overline\partial\zeta_0}{\CK{\zeta_0}}&=4\cwdot\int_S \Kill{\check\pi_0(\zeta_0\ww\phi)\ww\phi^\dagger)}{\CK{\zeta_0}}\cr
    &=4\cwdot\int_S \Kill{(\zeta_0\ww\phi)\ww\phi^\dagger}{\CK{\zeta_0}}\cr
     &=4\cwdot\int_S \Kill{\zeta_0\ww\phi}{\CK{\zeta_0}\ww\phi^\dagger}\cr
     &=-4\cwdot\int_S \Kill{\zeta_0\ww\phi}{\CK{\zeta_0\ww\phi}},\cr
    \end{align}
    where we used that $\rho(\phi)=-\phi^\dagger$ in the last equality and 
 Equation \eqref{killexchange} just before.
Thus after an integration by part we obtain

    \begin{align}
    \frac{1}{4}\int_S \Kill{\overline\partial\zeta_0}{\CK{ \overline\partial\zeta_0}}&=\frac{1}{4}\int_S \Kill{\overline\partial\zeta_0}{\partial\CK{\zeta_0}}\cr
    &=-\frac{1}{4}\int_S \Kill{\d^\nabla\overline\partial\zeta_0}{\CK{\zeta_0}}\cr
    &=\int_S \Kill{\zeta_0\ww\phi}{\CK{\zeta_0\ww\phi}}.
    \end{align}
  But, by Proposition \ref{Sign},

    \begin{align}
0   \leq  \frac{i}{4}\cwdot\int_S \Kill{\overline\partial\zeta_0}{\CK{\overline\partial\zeta_0}}
       =i\cwdot\int_S \Kill{\zeta_0\ww\phi}{\CK{\zeta_0\ww\phi}}\leq 0.
    \end{align}  
Thus, $\zeta_0\ww\phi=0$ and $\overline\partial\zeta_0=0$. Since $\CK{\zeta_0}=-\zeta_0$. It follows that 
$$
0=\CK{\overline\partial\zeta_0}=\partial\CK{\zeta_0}=-\partial\zeta_0.
$$
Thus $\partial\zeta_0=0$. It follows that $\d^\nabla\zeta_0=0$.

\end{proof}

\subsubsection{Proof of Proposition \ref{IFI0}}

Recall that we have the  decomposition
$$
\omega(\xi)=\zeta_0+\zeta_1+\zeta+\zeta^\dagger.
$$
We assume in this section that there exists a simple root $\alpha$, so that the
component $\zeta_\alpha=\omega_\alpha(\xi)$ of $\xi$ along $\mc G_\alpha$ vanishes (although the first proposition does not uses this hypothesis). 

\begin{proposition}
  We have $\zeta_0=0$ and $\zeta_1=0$.
\end{proposition}

\begin{proof} From Equation \eqref{IFIp1} of Proposition \ref{IFIp},
  we have that $\nabla\zeta_1=0$. Observe that
 $$
 \zeta_1=\sum_{\gamma\in\Delta\setminus Z\cup Z^\dagger}\zeta_\gamma,\hbox{ where }
 \zeta_\gamma\in\mc G_\gamma.
 $$
 Since the line bundles $\mc G_\gamma$ are parallel, it follows that
 $\nabla\zeta_\gamma=0$ for all $\gamma$ not in $Z$ nor in $Z^\dagger$. Recall that $\mc G_\alpha$
 is identified for all simple roots as $\mc K^{-1}$ over $\Sigma$ a complex line thanks to $\phi_\alpha$. It then follows that $\mc
 G_\beta$ is identified  with  $\mc K^{-\deg(\beta)}$  for any
 root $\beta$. Then since $\Sigma$ is not a torus,   $\mc K^{-\deg(\gamma)}$ is non trivial and   $\zeta_\gamma$
 vanishes at some point, hence everywhere since it is parallel.

By Proposition \ref{IFIp22},  $\zeta_0\ww\phi=0$, thus we get that for all
 simple root $\alpha$, $\zeta_0\ww\omega_\alpha=0$. By Definition \ref{cyc-map} Property \eqref{def:cyc6}, $\omega_\alpha$ never vanishes. Thus  $\alpha(\zeta_0)=0$. Since the simple roots form a basis of $\mk h^*$, $\zeta_0=0$.
\end{proof}

It follows from the previous proposition that
\begin{align}
  \xi=\zeta+\zeta^\dagger.
\end{align}
It remains to prove that $\zeta=0$ as well as $\zeta^\dagger=0$.  We
now split the proof in two cases.

\subsubsection{First case: $\ms G=\ms{SL}(3,\mathbb R)$}

Remark that in this case we have three positive roots, two simple that
we name $\alpha$, $\beta$ and one long $\eta=\alpha+\beta$. Also,
$\check\pi_0\not=0$ and $\pi_1=0$.
\begin{proposition}
  Assume that $\ms G=\ms{SL}(3,\mathbb R)$, then $\zeta=0$ and
  $\zeta^\dagger=0$.
\end{proposition}

\begin{proof}

  Let $\alpha$ and $\beta$ be the two simple roots and
  $\eta=\alpha+\beta$ be the longest root. Let us choose locally on $\Sigma$ a  Chevalley frame $\{\rm x_\alpha\}_{\alpha\in\Delta}$ such that $\rho(\rm x_\alpha)=x_{-\alpha}$, then we write   \begin{align}
    \phi^\dagger&=\psi_\eta\cwdot{\rm x}_{-\eta}+\psi_\alpha\cwdot {\rm
      x}_{\alpha}+\psi_\beta\cwdot {\rm x}_{\beta},\cr
       \phi&=\overline{\psi_\eta}\cwdot{\rm x}_{\eta}+\overline{\psi_\alpha}\cwdot {\rm
      x}_{-\alpha}+\overline{\psi_\beta}\cwdot {\rm x}_{-\beta},\cr
    \zeta&=\mu_\eta\cwdot{\rm
      x}_{\eta}+\mu_\alpha\cwdot {\rm
      x}_{-\alpha}+\mu_\beta\cwdot {\rm x}_{-\beta}.
  \end{align}
  Our hypothesis in that Section is that $\mu_\alpha=0$. Observe now
  that
  \begin{align}
    \zeta\ww\phi^\dagger&=-\mu_\eta.\psi_\eta {\rm
      h}_\eta+\mu_\beta.\psi_\beta {\rm
      h}_\beta+\mu_\alpha.\psi_\alpha {\rm
      h}_\alpha.\label{sl31}
  \end{align}
By Proposition \ref{torus} Property \eqref{Heta}, $\mk t_{\mathbb C}$ is generated by ${\rm h}_\eta$. Let us write
$$
{\rm h}_\beta=u+\lambda\cwdot {\rm h}_\eta,
$$
with $\Kill{u}{{\rm h}_\eta}=0$, so that $\check{\pi}_0(u)=u$. Observe also that $\Kill{{\rm h}_\beta}{u}=\Kill{u}{u}\not=0$, since $h_\beta$ is not proportional to $h_\eta$. 
  Thus  the equalities $\check\pi_0(\zeta\ww\phi^\dagger)=0$ and $\mu_\alpha=0$ imply that 
  \begin{align*}
  0=\braket{\zeta\ww\phi^\dagger|u}=\mu_\beta\psi_\beta\cwdot
  \braket{{\rm h}_\beta|u}.
\end{align*}
Since  $\psi_\beta$ never vanishes, it follows that $\mu_\beta=0$. Thus $\zeta\in\mc
G_{\eta}$. 

Since  by the reality condition \eqref{realcond2}, $\zeta^\dagger=\lambda(\zeta)$, it follows that 
$\zeta^\dagger=\mu_{\eta}\cwdot {\rm x_{-\eta}}\in\mc G_{-\eta}$ and 
\begin{align}
  \zeta^\dagger\ww\phi^\dagger &=\mu_{\eta}\psi_{\alpha}[{\rm x}_{-\eta},{\rm x}_\alpha]+\mu_{\eta}\psi_{\beta}[{\rm x}_{-\eta},{\rm x}_\beta]=\pi(\zeta^\dagger\ww\phi^\dagger).
\end{align}
Then the component of $\phi\ww\pi(\zeta^\dagger\ww\phi^\dagger)$ along $\mc G_{-\eta}$ is 
\begin{align}
\mu_\eta\cwdot \left(\overline{\psi_\alpha}\wedge \psi_\alpha\cwdot
    [{\rm x}_{-\alpha},[{\rm x}_{-\eta},{\rm x}_\alpha]] + \overline{\psi_\beta}\wedge \psi_\beta\cwdot
    [{\rm x}_{-\beta},[{\rm x}_{-\eta},{\rm x}_\beta]]\right).\label{sl3c3}
\end{align}
By Proposition \ref{phiphi},
we have
\begin{align}
  \nabla\zeta\ww\phi&=2\cwdot\phi\ww\pi(\zeta^\dagger\ww\phi^\dagger)
\end{align}
Since $\zeta\in \mc G_\eta$,  the component on $\mc G_{-\eta}$ of $\nabla\zeta\ww\phi$ is zero. Thus we obtain
\begin{align*}
0=\mu_\eta\cwdot \left(\overline{\psi_\alpha}\wedge \psi_\alpha\cwdot
    [{\rm x}_{-\alpha},[{\rm x}_{-\eta},{\rm x}_\alpha]] + \overline{\psi_\beta}\wedge \psi_\beta\cwdot
    [{\rm x}_{-\beta},[{\rm x}_{-\eta},{\rm x}_\beta]]\right).
\end{align*}
We can use an element $\mk w$ of the Weyl group that fixes $\eta$
and exchanges $\alpha$ and $\beta$, we then get that
\begin{align*}
  \mk w([{\rm x}_{-\alpha},[{\rm x}_{-\eta},{\rm x}_\alpha]])&=[{\rm x}_{-\beta},[{\rm x}_{-\eta},{\rm x}_\beta]]\end{align*}
On the other hand, $[{\rm x}_{-\alpha},[{\rm x}_{-\eta},{\rm x}_\alpha]]$ is along ${\rm x}_{-\eta}$ and thus fixed by $\mk w$. It follows that 
\begin{align*}
[{\rm x}_{-\alpha},[{\rm x}_{-\eta},{\rm x}_\alpha]]&=[{\rm x}_{-\beta},[{\rm x}_{-\eta},{\rm x}_\beta]]\not=0.
\end{align*}
Thus Equation \eqref{sl3c3} yields
\begin{align*}
0=\mu_\eta\cwdot \left(\overline{\psi_\alpha}\wedge \psi_\alpha + \overline{\psi_\beta}\wedge \psi_\beta\right).\label{sl3c4}
\end{align*}
Since $\overline{\psi_\alpha}\wedge \psi_\alpha$ and $\overline{\psi_\beta}\wedge \psi_\beta$ are both of type
$(1,1)$ and positive, we get that $\mu_\eta=0$.
We have finished proving that $\zeta^\dagger=\zeta=0$.
\end{proof}

\subsubsection{The general case}

Now we assume that $\ms G\not=\ms{SL}(3,\mathbb R)$. Then we prove

\begin{proposition}
  Assume that $\ms G\not=\ms{SL}(3,\mathbb R)$, then $\zeta=0$ and
  $\zeta^\dagger=0$.
\end{proposition}
\begin{proof}
  From Equation \eqref{IFIp2}
  \begin{align}
    \pi_1(\zeta\ww\phi)&=0.
  \end{align}
  Let $\alpha_0$ and $\beta_0$ be a pair of simple roots such that
  $\gamma=\alpha_0+\beta_0$ is a  (positive) root. Recall that $\ms
  G\not=\ms{SL}(3,\mathbb R)$ and thus $\gamma\not=\eta$. Then by the
  previous equation, the component of $v$ of $\zeta\ww\phi$ along $\mc
  G_\gamma$ is zero. But
$$
v=\sum_{\alpha,\beta\in
  Z\mid\alpha+\beta=\gamma}(\zeta_\alpha\ww\phi_\beta+\zeta_\beta\ww\phi_\alpha).
$$
However $\alpha-\eta$ is not a positive root when $\alpha$ is a simple root 
and since every positive root can be written uniquely as a sum of
simple roots, it follows that
 $$
 0=v=\zeta_{\alpha_0}\ww\phi_{\beta_0}+{\zeta_{\beta_0}}\ww{\phi_{\alpha_0}}.
 $$
 Thus for every pair of simple roots $\alpha$ and $\beta$ so that
 $\alpha+\beta$ is a
 root, $$\zeta_\alpha\ww\phi_\beta=-\zeta_\beta\ww\phi_\alpha.$$ Since
 the Dynkin diagram is connected and the $\phi_\alpha$ are not
 vanishing, since $\zeta_\alpha=0$ for some simple root, then
 $\zeta_{\beta}=0$ for every simple root $\beta$.
 
 Finally, we have obtained that $\zeta=\zeta_\eta\in \mc
 G_{\eta}$. Using Equation \eqref{IFIp2} again, we obtain that
 $\pi_1(\zeta_\eta\ww\sum_{\alpha\in\Pi}\phi_{-\alpha})=0$. Since $\ms
 G\not=\ms{SL}(3,\mathbb R)$, there exist a simple root $\alpha$ so that $\gamma\defeq\eta-\alpha$ is a positive root not in
 $Z\cup Z^\dagger$. Thus taking the
 component along $\mc G_\gamma$ one gets that
 $\zeta\ww\phi_{-\alpha}=0$. By the injectivity of $\phi_\alpha$,
 $\zeta=0$. A symmetric argument yields
 $\zeta^\dagger=0$.
\end{proof}

\section{Properness and the final argument}\label{sec:Proper}
Let  $\ms G_0$ be a split real simple group of rank 2 and $m_2$
 the degree of its longest root.  Let $\mc E_{p}$ be total space of
the vector bundle over Teichmüller space whose fibre at a complex
structure $J$ is $H^0(\Sigma, \mc K^{p})$.

Let now $\Psi$ be the map which associates to $(J,\rm q)$ in $\mc E_{m_2+1}$, the representation associated to the Higgs bundle
$(\widehat{\mc G},\Phi_q)$ as in Paragraph \ref{HitchSec} and Section \ref{sec:IR}.
Our main result is now

\begin{theorem} The map $\Psi$ is diffeomorphism.
\end{theorem}

This result has Theorem \ref{MainA} and Theorem \ref{MainB} as
immediate consequences.

\subsection{A theorem in differential calculus}

Let $\pi:P\to M$ be a smooth fibre bundle. Assume that the fibre $P_m\defeq
\pi^{-1}(m)$ at any point $m$ of $M$ is connected. Let $F$ be a smooth
positive function from $P$ to $\mathbb R$ and let $F_m=\left.F
\right\vert{P_m}$. Let finally $N\subset P$ be the set of critical points of $F_m$ for all $m$:
$$
N=\{x\in P\mid \d_x F_{\pi(x)}=0\}.
$$
Assume that
\begin{enumerate}
\item For all $m$, $F_m$ is proper,
\item $N$ is a connected submanifold of $P$ everywhere transverse to the fibres of $\pi$.
\end{enumerate}

Then we have

\begin{theorem}\label{calculus}
  Assume the above hypothesis, then 
    \begin{itemize}
  \item $\pi$ is a diffeomorphism of $N$ into $M$,
  \item $F_m$ has a unique critical point which is an absolute
    minimum.
  \end{itemize}
\end{theorem}

\subsubsection{Some preliminary lemmas}

Let $f$ be a smooth function defined on a manifold $Q$ diffeomorphic
to a closed ball. Let $m$ be a critical point of $f$ on $
Q$. Assume that $f$ has no critical point in $Q\setminus
\{m\}$.

The first lemma is obvious,

\begin{lemma}\label{C11} Assume that $m$ is a local minimum of $f$. There exists
  an neighborhood $V$ of $f$ in the $C^1$-topology, such that if $g\in
  V$, and $g$  has a at most one critical point in $Q$, then this critical
  point is a local minimum.
\end{lemma}

The second lemma is the following

\begin{lemma}\label{C12} Let $\{f_n\}_{n\in\mathbb N}$ be a sequence of functions
  converging to $f$ in the $C^1$ topology. Assume that for every $n$,
  $f_n$ has a unique critical point $m_n$ in $Q$. Assume furthermore
  that $m_n$ is a local minimum for $f_n$ and assume that $\{m_n\}_{n\in\mathbb
    N}$ converges to $m$. Then $m$ is a local minimum of $f$.
\end{lemma}

Let us explain the intuition behind this lemma. For simplicity assume $f_n(m_n)=0$. In the nice case when $f_n$ are all constant on the boundary, then the result is obviously true: $m_n$ are global minima and thus so is $m$. We want to reduce to this case by finding a small --but not too small-- neighbourhood of $m_n$ bounded by a level set. In other words, find $\epsilon_n$ positive so that
\begin{enumerate}
	\item  all $\epsilon_n$ are uniformly bounded away from $0$ by some $\epsilon_0$ (produced by the first step below).	\item A connected component of the level set $L_n=f_n^{-1}(\epsilon_n)$ bounds an set $O_n$ which contains  $m$ and does not touch the boundary.\label{813}
	\end{enumerate}
Then using the initial remark and replacing $Q$ by $O_n$ yields the result. In order to obtain a lower bound on $\epsilon_n$, we consider the maximum of those $\epsilon_n$ which satisfies \eqref{813}. It follows that we have a gradient line $\gamma_n$ of $f_n$ joining $m_n$ to the boundary whose value by $f_n$ is $[0,\epsilon_n]$. If $\{\epsilon_n\}_{n\in\mathbb N}$ tends to zero, then by taking limits of $\gamma_n$, one obtains a subset on which $f$ is constant and equal to zero and actually consists only in critical points. This would be the contradiction.

Let us now proceed carefully to the actual proof.

\subsubsection{Proof of Lemma \ref{C12}: first steps}  For simplicity we may as well assume that for all $n$,
  $f_n(m_n)=0$.  Choose an auxiliary Riemannian metric. Let $(\phi_t^n)_{t\in\mathbb R}$ be the gradient flow of $f_n$. 

Let $Z_n$ be the set of those points $z$ in $\partial Q$ such that an $f_n$-gradient line issued from $m_n$ hits $z$ (after possibly touching the boundary before). Let $B_n$ be a small closed neighbourhood of $m_n$, such that the inverse gradient line of any point in $B_n$ converges to $m_n$.
\begin{proposition}
The set $Z_n$ is not empty
\end{proposition}

\begin{proof} Let $x\not=m_n$ be a point in $B_n$. Let $t\to \gamma_n(t)$ be the gradient flow of $x_n$. After a finite time $t_0$, $\gamma_n(t)$ leaves $B_n$ and never come back afterwards. Since $f_n$ has no critical point outside $B_n$, the norm of $\dt\gamma_n(t)$  is uniformly  bounded from below for $t>t_0$. Since $f_n$ is bounded, it follows that $\gamma_n(t)$ hits the boundary after a finite time. Thus $Z_n$ is not empty.
\end{proof}

\begin{proposition}
The set  $Z_n$ is closed.
\end{proposition}
\begin{proof}
Since the gradient of $f_n$ is bounded from below on $Q\setminus B_n$ and $f$ is bounded, there exist some positive $T$ so that for all $z\in Z_n$, $\phi^n_{-T}(z)\in B_n$. Thus if a sequence $\{z_p\}_{p\in\mathbb N}$ of points of $Z_n$ converges to $z$, it follows by continuity that $\phi^n_{-T}(z)\in B_n$. In particular $z\in Z_n$. Thus $Z_n$ is closed.\end{proof}
\begin{proposition}\label{pro813:1}
There exists $\epsilon_0$ so that 
\begin{align}
  \forall n, \forall z\in Z_n,\,\, f_n(z)>\epsilon_0>0.
\end{align}
\end{proposition}
\begin{proof}
Let us work by contradiction and assume that
we can find a sequence $\{z_n\}_{n\in\mathbb N}$ with  $z_n\in Z_n$ so that $f_n(z_n)\rightarrow 0$. Let
$\gamma_n$ be the gradient line of $f_n$ issuing from $z_n$. We choose a subsequence so that
  $\{z_n\}_{n\in\mathbb N}$ converges to a point $z\not=m$. Let $\gamma$ be a Hausdorff limit of a subsequence of the sequence of closed connected sets $\{W_n\defeq\overline{\gamma_n[-\infty,0]}\}_{n\in\mathbb N}$. Then the hypothesis of the paragraph implies 
  \begin{enumerate}
  \item the function $f$ is constant and equal to zero on the connected set $\gamma$: indeed $f_n(W_n)\subset [0,f_n(z_n)]$.
  \item  let $\phi_t$ be the gradient flow of $f$, then for all positive $t$, we also have $\phi_{-t}(\gamma)\subset\gamma$.
  \end{enumerate}
From both previous facts, it follows that all points in the connected set $\gamma$ -- and in particular $z\not=m$-- are critical, which is a contradiction with our hypothesis on $f$.
\end{proof} 
\subsubsection{Proof of Lemma \ref{C12}: last steps} 
Let $\gamma_n$ be an orbit of the gradient of $f_n$ that connects $m_n$ to a point $z_n$ in $\partial Q$.
 Let $0<\alpha<\epsilon_0$ where $\epsilon_0$ is obtained from Proposition \ref{pro813:1}. Let $L^n_\alpha=f^{-1}_n(\alpha)$. Observe that since $f_n(z_n)>\alpha$,  $L^n_\alpha$ intersects $\gamma_n$ in an unique point $q_n$. Let  $S^n_\alpha$ be the
  connected component of $L^n_\alpha$ containing $q_n$.
\begin{proposition}
For all  $\alpha$ less than $\epsilon_0$, $S^n_\alpha$ is a closed submanifold of $Q\setminus\partial Q$ bounding an open set 
 $B^n_\alpha$ containing $m_n$. 
\end{proposition}
\begin{proof} 
Since $m_n$ is a local minimum for $f_n$ there exists $\beta$ (depending on $n$) with $0<\beta<\epsilon_0$ such that
all
  gradient lines passing though $S^n_\beta$ end up at $m_n$. Then $S^n_\beta\cap\partial Q \subset Z_n$ but since $\beta< \epsilon_n=\inf\{f_n(z)\mid z\in Z_n\}$ it follows that S$^n_\beta\cap\partial Q=\emptyset$. Thus $S^n_\beta$ is a closed submanifold of $Q\setminus\partial Q$ and bounds an open set $B^n_\beta$ containing $m_n$. 
  
We may choose an auxiliary Riemannian metric (depending on $n$) so that the norm of the gradient of $f_n$ is 1 outside $B^n_\beta$. It follows that if $$t<\epsilon_0-\beta,$$
then $\phi^n_t(S^n_\beta)$ is  closed submanifold of $Q\setminus\partial Q$, intersecting $\gamma_n$ and on which $f_n$ is equal to $t+\beta$. Thus 
\begin{align}
    \phi^n_t(S^n_\beta)=S^n_{t+\beta}.
\end{align}
 It follows that for all  $\alpha$ less than $\epsilon_0$, $S^n_\alpha$ bounds an open set 
 $B^n_\alpha$ containing $m_n$. \end{proof}

\begin{proposition} 
There exists an open set $O$, with $\overline O\subset Q\setminus \partial Q$, such  that $f$ is constant on $\partial O$.
\end{proposition}

\begin{proof}  Since the gradient of $f_n$ is uniformly bounded from above there exists some positive $\beta$ such that for $n$ large enough
\begin{align}
    d\left(S^n_{\epsilon_0/2},S^n_{\epsilon_0/4}\right)>\beta  .
\end{align}
In particular 
\begin{align}
    d\left(\partial Q,S^n_{\epsilon_0/4}\right)>\beta  .
\end{align}

 It follows that $\{S^n_{\epsilon_0/4}\}_{n\in\mathbb N}$ converge to a connected component $S_{\epsilon_0/4}$ of a level set  of $f$, which  is a closed submanifold in
  $Q\setminus\partial Q$, which thus bounds an open set $O$. \end{proof}

This last proposition implies Lemma \ref{C12}: $f$ has a minimum on $O$ which has to be $m$ since $f$ has a unique critical point in $Q$. Thus $m$ is a local minimum for $f$.

\subsubsection{Proof of Theorem \ref{calculus}} By assumption $N$ is a
closed connected submanifold transverse to the fibres.

For any point $m$ in $N$, the transversality hypothesis implies that we can find neighbourhoods $U$ of $\pi(m)$, $W$ of $m$ so that we can identify $W$ with $U\times Q$, where $Q$ is diffeomorphic to an open ball and $\pi:U\times Q\to U$ is the projection on the first factor.

Let $X$ be the subset of
those $x$ in $N$ such that $x$ is a local minimal point of $F_{\pi(x)}$. Then $X$ is non empty since $F_m$ is proper. 

Then using the neighborhood $U$ and $W=U\times Q$ as above, we obtain that $X$
is open by Lemma \ref{C11} and  closed by Lemma \ref{C12}. By connectedness, $X=N$.

Thus every critical point of $f_m$ is a local minimum. Since $f_m$ is
proper positive and $P_m$ is connected, elementary Morse theory tells us that $f_m$ has a
unique critical point which is an absolute minimum. In particular, the
local diffeomorphism $\pi$ from $N$ to $M$ is injective and
surjective, thus a global diffeomorphism.

\subsection{Proof of the main Theorem}

Let $P=\mc T\times \mc H(\Sigma,\ms G_0)$ where $\mc T$ is Teichmüller
space and $\mc H(\Sigma,\ms G_0)$ is the Hitchin component of $\ms
G_0$. Let $\pi$ be the projection on the second factor. Let $F$ be the
function which associates to every $(J,\delta)$ in $P$ the energy of the
unique $\delta$-equivariant harmonic map in the
symmetric space $\ms S (\ms G_0)$  of $\ms G_0$. By \cite{Labourie:2005a}, $F$ is smooth
positive and $F_\delta$ is a proper map. By \cite{Sacks:1981}, \cite{Sacks:1982} and \cite{Schoen:1979}, the critical points of $F_\delta$ are equivariant  minimal mappings.

Finally by Hitchin's fundamental result in \cite{Hitchin:1992es}, $P$ is diffeomorphic to the bundle over $\mc T$ whose fibre at every point is
$\mc E_2\oplus\mc E_{m_2+1}$. From Theorem \ref{IFI0co} the map $\widehat\Psi:
(J,{\rm q})\mapsto (J,\Psi(J,{\rm q}))$ is transverse to the fibre.
Moreover $\Psi$ is an embedding and its image  $N$ is precisely the pairs $(J,\delta)$ such that the $\delta$-equivariant harmonic mapping $f$ from $\Sigma$ equipped with $J$ has a vanishing Hopf differential, that is $f$ is minimal. Thus $N$ 
satisfy the conditions of the Theorem \ref{calculus}.

The main Theorem then follows.

\section{The Kähler structures}\label{sec:K}

We state a more precise result using the notation of the introduction.
For any integer greater than 1, let $\mc E_m$ be the holomorphic vector
bundle over Teichmüller space whose fibre at a Riemann surface
$\Sigma$ is
\begin{align*}
  \left(\mc E_m\right)_\Sigma &\defeq H^0\left(\Sigma,{\mc
      K}^{m}\right).
\end{align*}
Let ${\rm m}=(m_1,\ldots,m_\ell)$, with all $m_i>1$. We denote by $\mc
E({\rm m})$ be the holomorphic vector bundle over Teichmüller space
whose fibre at a Riemann surface $\Sigma$ is
\begin{align*}
  \mc E({\rm m})_\Sigma &\defeq\bigoplus_{i=1}^{\ell}\left(\mc E_{m_i}\right)_\Sigma.
\end{align*}

\begin{proposition} \label{prop:Ka} The complex vector bundle $\overline{\mc E({\rm m})}$ carries a holomorphic structure which make it isomorphic to $\mc E^*{(\rm m})$ and a compatible $p$-dimensional family of 
  Kähler structures with the following properties
  \begin{enumerate}
 \item  The restriction of the Kähler structures is a multiple of  $L^2$-metric in every fibre,
  \item The Kähler structure is invariant by the mapping class group
    action,
  \item The zero section is totally geodesic,
  \item The metric induced on the zero section is Weil--Petersson
    metric.
  \end{enumerate}
\end{proposition}
The $L^2$-metric on $H^0(\Sigma,{\mc K}^m)$ is taken with respect of
the hyperbolic metric on $\Sigma$.

We explain the construction of Kim and Zhang in \cite{Kim:2013wc} given in
the cubic case which extends with only slight modifications to the
general case.  We reproduce the proof here, with some small
simplifications, in order to get more specific details on the Kähler metrics that we construct.

\subsection{Positive Hermitian bundles}

Let $E\to X$ be a holomorphic bundle over a complex manifold $X$
equipped with a Hermitian metric $h$. Let $\nabla$ be the Chern
connection and $R^\nabla$ be the Chern curvature that we see
as a tensor element of $\Omega^2(\T_{\mathbb C} X,\End(E))$. We then define a
2-tensor $\Theta$ on $\T_{\mathbb C}X\otimes\End(E)$ by
\begin{align}
  \Theta(Y\otimes u, Z\otimes v)&\defeq i\cwdot
  h\left(R^\nabla\left(Y,Z\right)\cwdot u,v\right).
\end{align}

Using the symmetry of the curvature tensor, one gets that $\Theta$ is
Hermitian quadratic.  We now say that (See
\cite{Griffiths:1965wc}, Definition 3.9 in \cite{Demailly:2012um})
\begin{itemize}
\item the Hermitian bundle $E$ is {\em Nakano positive} if $\Theta$ is
  Hermitian positive,
\item the Hermitian bundle $E$ is {\em Griffiths positive} if for all
  non zero decomposable vectors $X\otimes u$, $\Theta(X\otimes
  u,X\otimes u)>0$,
\item the Hermitian bundle $E$ is {\em Griffiths negative} if for all
  non zero decomposable vectors $X\otimes u$, $\Theta(X\otimes
  u,X\otimes u)<0$,
\end{itemize}
From the definitions one immediately gets the following facts.
\begin{enumerate}
\item a Nakano positive bundle is Griffiths positive,
\item the sum of Nakano positive is Nakano positive,
\item the dual of a Griffiths positive bundle is Griffiths negative,
\end{enumerate}
As an easy consequence, the following seems well known to complex geometers.
\begin{proposition}\label{pro:KZ} {\sc [Kähler Metric on the total
    space]}

  Let $\pi:\mc E\to X$ be a holomorphic bundle over a complex manifold
  equipped with a Griffiths negative Hermitian metric $h$ that we
  consider as a fibrewise quadratic function. Assume that $\mc E$ is
  equipped with a holomorphic action of some group $\Gamma$
  preserving the Hermitian metric $h$ and a Kähler metric $g$ on
  $X$. Then for any $\epsilon>0$,
$$
H:=\epsilon\cwdot\partial\overline{\partial}h+\pi^*g,
$$ 
is a $\Gamma$-invariant Kähler metric on $\mc E$. Furthermore
\begin{enumerate}
\item $H$ is linear along the fibre,
\item the zero section is a totally geodesic isometric immersion.
\end{enumerate}

\end{proposition}
Compare with the content of the proof of Theorem 5.4 in
\cite{Kim:2013wc}.
\begin{proof} Let $\sigma$ be a holomorphic section of $\mc E$.  A
  classical computation (Proposition 3.1.5 of \cite{Kobayashi:1987ub})
  says that $\partial\overline{\partial}\left(h(\sigma)\right)$ is
  positive. Thus $h$ is a plurisubharmonic function on $\mc E$.  Since
  $\epsilon\cwdot\partial\overline{\partial}h$ is positive along the
  fibres, we get that $H$ is a Kähler metric. Since $u\mapsto -u$ is
  an isometry of $H$ whose fixed point is the zero section, the zero
  section is totally geodesic. The other statements are obvious by
  construction.
\end{proof}

\subsection{Pushforward bundles} Theorem 1.2. in Bo Berndtsson
\cite{Berndtsson:2009hr} is a powerful way to assert the positivity of
bundles.

\begin{theorem}{\sc [Berndtsson]}\label{theo:bo}
  Let $\pi:X\to Y$ be a holomorphic fibration with non singular and
  compact fibres. Assume $X$ is Kähler. We denote by $X_y$ the fibre
  over $y\in Y$. Let $\mc L$ be a positive line bundle on $Y$.  Then
  the vector bundle over $Y$ whose fibre at $y$ is
$$
H^0(X_y, \mc L\otimes {\mc K}_{X/Y}),
$$
equipped with the $L^2$-metric, is Nakano positive.
\end{theorem}

As a consequence, we obtain as in \cite{Kim:2013wc}
\begin{proposition}{\sc [Inkang Kim--Genkhai Zhang]}
  The holomorphic bundle $\mc E_m$ equipped with the $L^2$ metric is
  Nakano positive.
\end{proposition}

\begin{proof}
  We apply Theorem \ref{theo:bo} to the following situation: $Y$ is
  Teichmüller space, $X$ is the Teichmüller curve, ${\mc L}= {\mc
    K}_{X/Y}^{m-1}$ is the canonical bundle of the fibre to the power
  $m-1$. Then at a Riemann surface $\Sigma$,
$$
H^0(X_y, {\mc L}\otimes {\mc K}_{X/Y})=H^0(\Sigma, {\mc
  K}_\Sigma^n)=(\mathcal E_n)_\Sigma.
$$
It remains to check the hypothesis. Indeed
\begin{itemize}
\item ${\mc L}$ is positive by Lemma 5.8 of \cite{Wolpert:1986ie},
\item $X$ is Kähler, as a consequence, being the base of a positive line bundle.
\end{itemize}
The result follows.
\end{proof}

Since the sum of Nakano positive bundles is Nakano positive, we immediately
get that the bundle $\mc E({\rm m})$ is Nakano positive, where ${\rm
  m}=(m_1,\ldots,m_\ell)$. Thus using our preliminary remarks, we have

\begin{proposition}\label{pro:neg}
  The bundle $\mc E^*({\rm m})$ is Griffiths negative.
\end{proposition}

\subsection{The Kähler property}
In order to get \ref{prop:Ka}, we apply Proposition \ref{pro:KZ} to
$\mc F= \mc E^*({\rm m})$ using Proposition \ref{pro:neg}.

Actually we have a family of Kähler metrics since we have a natural
holomorphic $\mathbb C^p$ action on $\mc E^*({\rm m})$.

By construction the metric is invariant under the mapping class group.

Since furthermore the metric is invariant by rotation in the fibres,
the zero section is totally geodesic. Furthermore, by construction the
zero section is an isometry from $Y$ into $\mc F$.
\section{Area rigidity}
For any split real rank 2 group $\ms G_0$,  let $c(\ms G_0)$ be the curvature of the totally geodesic hyperbolic plane associated to the principal $\ms{SL}_2$ in $\ms G_0$. Our goal is to prove in this section the following
\begin{theorem}\label{area-rigid}
Let $\delta$ be a Hitchin representation of $\pi_1(\Sigma)$ in $\ms G_0$ where $\ms G_0$ has rank 2.
Then
$$
\operatorname{MinArea}(\delta)\geq c(\ms G_0)\cwdot \chi(S),
$$
with equality only if $\delta$ is Fuchsian.
\end{theorem}

\subsection{Forms}
Using the notation given in the decomposition \eqref{MC-dec}, let us consider for any $a\in\mk h^*$ the 2-form
\begin{align}
\Omega_a\defeq\braket{a \mid\omega\ww\omega}\in\Omega^2(\X).
\end{align}
Then we have the following result
\begin{proposition}\label{omegatopo}
For any $a\in\mk h^*$, the form $\Omega_a$ is closed. Moreover, for any $\delta$ in the Hitchin component, let $\Sigma_\delta$ be the unique minimal surface in $\X/\delta(\pi_1(\Sigma))$ then 
$
\int_{\Sigma_\delta}\Omega_a$
does not depend on $\delta$.
\end{proposition}
\begin{proof} 
By equation \eqref{curvfund}, 
$$
\omega\ww\omega= -R^\nabla-\d^\nabla\omega.
$$
Thus 
\begin{align}
-\Omega_a&=\braket{a\mid \d^\nabla\omega}+\braket{a\mid R^\nabla}\cr
&=\d \braket{a\mid \omega}+\braket{a\mid R^\nabla}.
\end{align}
Thus
$\Omega_a$ is closed (by the Bianchi identity) and in the same cohomology class as $-\braket{a\mid R^\nabla}$. Since $R^\nabla$ is the curvature of the $\ms T$-bundle $\ms G\to\ms G/\ms T$, choosing $a$ to be integral,
 it follows that there exists a constant $b         (\ms G_0)$  and an $S^1$-bundle $P_a$ over $\X$ so that $\frac{1}{b(\ms G_0)}\Omega_a$  is the curvature of $P_a$. In particular
$$
f_a(\delta)\defeq \frac{1}{b(\ms G_0)}\int_{\Sigma_\delta}\Omega_a\in\mathbb Z.
$$
Since $f_a(\delta)$ depends continuously on $\delta$, $f_a(\delta)$ is constant. The result extend by linearity to all $a$. \end{proof}

Let us now consider the following 2-forms on $\X$ (using the convention of paragraph \ref{prel-form})
\begin{align}
\Omega_0&\defeq i\cwdot\sum_{\alpha\in\Pi}\Kill{\omega_{-\alpha}}{\omega_{\alpha}}\cr
\Omega_1&\defeq i\cwdot\Kill{\omega_{\eta}}{\omega_{-\eta}}.
\end{align}
Then we have
\begin{proposition}\label{pro:area}
The form $\Omega_0$ and $\Omega_1$ are positive on any cyclic surface. Moreover
if $\Sigma$ be a cyclic surface in $\X_0$ and $p$ is the projection of $\X_0$ to $\ms S(G_0)$,
then
$$
\int_{\Sigma}\Omega_0+\Omega_1=\operatorname{Area}(p(\Sigma)).
$$
\end{proposition}
\begin{proof} Recall that for a cyclic surface  $\omega_\eta$ and $\omega_{-\alpha}$ are of type $(1,0)$ for $\alpha\in\Pi$ and $\omega_{-\beta}=-\rho(\omega_\beta)$. Thus the positivity of $\Omega_0$ and $\Omega_1$ on cyclic surfaces  follows from Proposition \ref{Sign}. From Equation \eqref{area-phi}, it follows
$$
\operatorname{Area}(p(\Sigma))=i.\int_\Sigma\Kill{\Phi}{\Phi^\dagger}. 
$$
Recall that
\begin{align}
 \Phi&=\omega_\eta+\sum_{\alpha\in\Pi}\omega_{-\alpha},\cr
 \Phi^\dagger&=\omega_{-\eta}+\sum_{\alpha\in\Pi}\omega_{\alpha}.
 \end{align}
Thus since $\mc G_\alpha$ and $\mc G_\beta$ are orthogonal with respect to the Killing form unless $\alpha+\beta=0$. It follows that
\begin{align*}
i\cwdot\Kill{\Phi}{\Phi^\dagger}&=i\cwdot\Kill{\omega_{\eta}}{\omega_{-\eta}}+i\cwdot\sum_{\alpha\in\Pi}\Kill{\omega_{-\alpha}}{\omega_{\alpha}}\cr 
&=\Omega_0+\Omega_1.
 \end{align*}
\end{proof}

We finally will need
\begin{proposition}\label{pro:u} Let $\Pi$ be the set of simple roots in $\ms G_0$. There there exists a unique element  $u_0$ in $\mk h^*$, such that 
for any simple root $\alpha$ and $X\in \mk g_{-\alpha}$, $Y\in \mk g_{\alpha}$, we have
$$
\Kill{X}{Y}=\braket{u_0\mid[X,Y]}.
$$
Moreover, there exists a positive constant $k_0$, so that if
$X\in \mk g_{-\eta}$, $Y\in \mk g_{\eta}$, we have
$$
k_0\cwdot \Kill{X}{Y}=\braket{u_0\mid[X,Y]}.
$$
\end{proposition}

\noindent{\em Proof:} Let us choose a Chevalley basis $\{\rm x_\alpha\}_{\alpha\in\Delta}$. Let us write $X=x\cwdot{\rm x_{-\alpha}}$, $Y=y\cwdot{\rm x_{\alpha}}$, then
 $$
 \Kill{X}{Y}=x\cwdot y\cwdot  \Kill{\rm x_\alpha}{\rm x_{-\alpha}}.
 $$
 On the other hand
$$
\braket{u_0\mid[X,Y]}=-x\cwdot y \cwdot\Kill{u_0}{{\rm h}_\alpha}.
$$ 
Thus $u_0$ is uniquely determined by
$$
\Kill{u_0}{{\rm h}_\alpha}=-\Kill{\rm x_\alpha}{\rm x_{-\alpha}}.
$$
We may choose a Cartan involution so that ${\rm x}_{-\alpha}=\rho({\rm x_\alpha})$. Thus $\Kill{\rm x_\alpha}{\rm x_{-\alpha}}<0$.
Since 
$$
{\rm h}_\eta=\sum_{\alpha\in\Pi} r_\alpha\cwdot {\rm h}_\alpha,
$$
with $r_\alpha>0$. It follows that
$$
\Kill{u_0}{{\rm h}_\eta}>0.
$$
In particular, if $X\in \mk g_\eta$ and $Y\in\mk g_{-\eta}$ we have
$$
k_0\cwdot\Kill{X}{Y}=\braket{u_0\mid[X,Y]},
$$
where
$$
k_0=-\frac{\Kill{u_0}{{\rm h}_\eta}}{\Kill{\rm x_\eta}{\rm x_{-\eta}}}>0.
$$
\qedhere
As a corollary of this proposition one obtains immediately
\begin{corollary}\label{cor:u}
Let $u_0$ be defined as in Proposition \ref{pro:u} and $v_0=i\cwdot u_0$, then
$$
\Omega_{v_0}=\Omega_0-k_0\cwdot\Omega_1.
$$
\end{corollary}
\subsection{Proof of the Area Rigidity Theorem}
We can now prove Theorem \ref{area-rigid}. From Corollary \ref{cor:u} and Proposition \ref{pro:area}, on gets that
\begin{align}
\operatorname{Area}(\pi(\Sigma_\delta))=\int_{\Sigma_{\delta}}\Omega_{v_0}+ (k_0+1)\int_{\Sigma_{\delta}}\Omega_1\,.
\end{align}
If $\delta_0$ is a Fuchsian representation, then the corresponding cyclic surface is Fuchsian (see Proposition \ref{fuchs}) and thus $\omega_\eta$ and  $\Omega_1$ vanish. Thus we get that if $\delta_0$ is Fuchsian
\begin{align}
\int_{\Sigma_{\delta_0}}\Omega_{v_0}
=\operatorname{Area}(\Sigma_{\delta_0})=c(\ms G_0)\chi(S)\,.
\end{align}
By Proposition \ref{omegatopo} $\int_{\Sigma_{\delta}}\Omega_{v_0}$ does not depend on $\delta$.
It thus follows that
\begin{align}
\operatorname{MinArea}(\delta)&=\operatorname{Area}(\pi(\Sigma_\delta))\cr
&=c(\ms G_0)\chi(S)+(k_0+1)\int_{\Sigma_\delta}\Omega_1\,.
\end{align}
The result now follows from the fact that $\Omega_1$ is positive on cyclic surfaces by Proposition \ref{pro:area}. Moreover $\Omega_1$ vanishes if and only if $\omega_\eta$ vanishes, but  $\omega_\eta$ vanishes if and only if the Higgs field takes values in $\sum_{\alpha\in\Pi}\mc G_\alpha$ -- that is $\delta$ is Fuchsian ( see Theorem \ref{th:H2} ).

\providecommand{\bysame}{\leavevmode\hbox to3em{\hrulefill}\thinspace}
\providecommand{\MR}{\relax\ifhmode\unskip\space\fi MR }
\providecommand{\MRhref}[2]{%
  \href{http://www.ams.org/mathscinet-getitem?mr=#1}{#2}
}
\providecommand{\href}[2]{#2}

\end{document}